\documentclass[suppldata]{interact}

\usepackage{epstopdf}% To incorporate .eps illustrations using PDFLaTeX, etc.
\usepackage[caption=false]{subfig}% Support for small, `sub' figures and tables
\usepackage{hyperref}
\usepackage[dvipsnames]{xcolor}
\usepackage{stackengine}
\usepackage{amsmath}   % Maths
\usepackage{amssymb}
\usepackage{comment}
\usepackage{textcomp}
\usepackage{dsfont}
\usepackage{todonotes}
\usepackage{todonotes}
\usepackage{algorithm}% http://ctan.org/pkg/algorithms
\usepackage{algpseudocode}% http://ctan.org/pkg/algorithmicx
\usepackage{mathtools}

\usepackage{tikz}

\definecolor{lime}{HTML}{A6CE39}
\DeclareRobustCommand{\orcidicon}{%
	\begin{tikzpicture}
	\draw[lime, fill=lime] (0,0) 
	circle [radius=0.16] 
	node[white] {{\fontfamily{qag}\selectfont \tiny ID}};
	\draw[white, fill=white] (-0.0625,0.095) 
	circle [radius=0.007];
	\end{tikzpicture}
	\hspace{-2mm}
}

\foreach \x in {A, ..., Z}{%
	\expandafter\xdef\csname orcid\x\endcsname{\noexpand\href{https://orcid.org/\csname orcidauthor\x\endcsname}{\noexpand\orcidicon}}
}

\usepackage{multirow}

\usepackage[numbers,sort&compress]{natbib}% Citation support using natbib.sty
\bibpunct[, ]{[}{]}{,}{n}{,}{,}% Citation support using natbib.sty
\renewcommand\bibfont{\fontsize{10}{12}\selectfont}% Bibliography support using natbib.sty

\theoremstyle{plain}% Theorem-like structures provided by amsthm.sty
\newtheorem{theorem}{Theorem}[section]
\newtheorem{lemma}[theorem]{Lemma}

\newtheorem{proposition}[theorem]{Proposition}

\theoremstyle{definition}
\newtheorem{definition}[theorem]{Definition}
\newtheorem{example}[theorem]{Example}

\theoremstyle{remark}
\newtheorem{remark}{Remark}
\newtheorem{assumption}{Assumption}

\begin{document}

%\articletype{ARTICLE TEMPLATE}% Specify the article type or omit as appropriate

\title{Global minimization of a minimum \\of a finite collection of functions}

\author{
\name{Guillaume Van Dessel\textsuperscript{a}\orcidA{}\thanks{CONTACT Guillaume Van Dessel. Email: guillaume.vandessel@uclouvain.be} and Fran\c{c}ois Glineur\textsuperscript{a,b}\orcidB{}}
\affil{\textsuperscript{a}UCLouvain, ICTEAM (INMA), 4 Avenue Georges Lema\^{\i}tre, Louvain-la-Neuve, BE; \\\textsuperscript{b}UCLouvain, CORE, 34
Voie du Roman Pays, Louvain-la-Neuve, BE}
}

\maketitle

\begin{abstract}
We consider the global minimization of a particular type of minimum structured optimization problems wherein the variables must belong to some basic set, the feasible domain is described by the intersection of a large number of functional \emph{constraints} and the objective stems as the pointwise minimum of a collection of functional \emph{pieces}. Among others, this setting includes all non-convex piecewise linear problems. The global minimum of a minimum structured problem can be computed using a simple enumeration scheme by solving a sequence of individual problems involving each a single \emph{piece} and then taking the smallest individual minimum. 

We propose a new algorithm, called Upper-Lower Optimization (\texttt{ULO}), tackling problems from the aforementioned class. Our method does not require the solution of every individual problem listed in the baseline enumeration scheme, yielding potential substantial computational gains. It alternates between (a) local-search steps, which minimize upper models, and (b) suggestion steps, which optimize lower models. We prove that \texttt{ULO} can return a solution of any prescribed global optimality accuracy. According to a computational complexity model based on the cost of solving the individual problems, we analyze in which practical scenarios \texttt{ULO} is expected to outperform the baseline enumeration scheme. Finally, we empirically validate our approach on a set of piecewise linear programs (PL) of incremental difficulty.

\end{abstract}


\vspace{-20pt}
\section{Introduction}
\label{intro_notes}

        Let $\mathcal{X}$ denote a subset of $\mathbb{R}^d$. We consider a collection of $n \in \mathbb{N}$ continuous objective \emph{pieces} $\big\{f^{(i)}\big\}_{i \in [n]}$%\footnote{For any natural number $B \in \mathbb{N}$, we use the shorthand $[B]$ to depict the set $\{1,\dots,B\}$.} 
        and a collection of $m \in \mathbb{N}$ functional \emph{constraints} $\big\{c^{(j)}\big\}_{j \in [m]}$.\\ We are concerned with \textit{minimum structured} (MS) optimization problems of the form
\begin{equation}
\begin{aligned}
F^* := \min_{x \,\in \,\mathcal{X}} \quad &  \min_{i=1,\dots,n}\,f^{(i)}(x)\label{eq:min_problem}
 \\
\textrm{s.t.} \quad & c^{(j)}(x) \leq 0 & \forall j \in [m]  \\ 
\end{aligned}
\tag{MS}
\end{equation}
for which we assume that the optimal value is finite, i.e. $-\infty<F^*<\infty$.  \\
The combinatorial nature of \eqref{eq:min_problem} in terms of both dimensions $(n,m)$ suggests the following two notations. For any subset $S \subseteq [m]$, we introduce
\begin{equation} 
\mathcal{D}^{(S)} := \Big\{x\in \mathbb{R}^d\,\Big|\,c^{(j)}(x) \leq 0 \quad \forall j \in S\bigg\},
\label{eq:domain_implicit} 
\end{equation}
the implicit domain defined by the intersection of the constraints indexed in $S$. \\Then, building upon this notation, we define for every $H\subseteq [n]$ the function
\begin{equation} 
x \to F^{(H,S)}(x) := \begin{cases} \min_{i \in H}\,f^{(i)}(x)& x \in \mathcal{D}^{(S)}\\ \infty & x \not \in \mathcal{D}^{(S)}. \end{cases} 
\label{eq:small_F} 
\end{equation}
$F^{(H,S)}$ represents a simplified version of the overall objective $F$ where, instead of considering the minimum over all the $n$ \emph{pieces} while satisfying each of the $m$ \emph{constraints}, one only cares about the minimum over \emph{pieces} indexed in $H$ while only taking into account \emph{constraints} in $S$.  
We can now rewrite more compactly,
\begin{equation} F^* = \min_{x \,\in\,\mathcal{X}}\,F(x):=F^{([n],[m])}(x).\label{eq:min_problem_equiv}
\end{equation}
Although non-convex and intrinsically hard, the \eqref{eq:min_problem} framework allows for building successive global approximations that are either upper models or lower models of the objective function. We will assume throughout this paper the availability of an oracle that can perform exact global minimization of any the aforementioned models $F^{(H,S)}$. More specifically, the oracle can solve any problem 
\begin{equation}
\nu(H,S) := \min_{x \in \mathcal{X}}\,F^{(H,S)}(x).
\label{eq:small_min_problem} \tag{Oracle}
\end{equation} % oracle CAN BE obtained
Note that it is sufficient to assume a black-box computing for any $i \in [n]$, 
\begin{equation}
\nu(\{i\},S) = \min_{x \in \mathcal{X}}\,f^{(i)}(x) \quad \text{s.t.}\quad \max_{j \in S}\,c^{(j)}(x) \leq 0
\label{eq:small_min_problem_i}
\end{equation}
since it trivially implies $\nu(H,S) = \min_{i \in H}\,\nu(\{i\},S)$.
Obviously, calculating $F^*$ can be achieved by calling $n$ times this black-box as in  
\begin{equation}
F^* = \min_{i \in [n]}\,\underbrace{\nu(\{i\},[m])}_{\,:=\,\nu^{(i)}}.
\label{eq:enumeration_scheme_ref} 
\end{equation}
In the absence of any further simplification, the time complexity of the above \emph{enumeration} strategy \eqref{eq:enumeration_scheme_ref} to solve \eqref{eq:min_problem} inevitably scales linearly with $n$. This might become prohibitive if $n$ is big and the efforts taken by the black-box to obtain values $\nu^{(i)}$ for every $i \in [n]$ are not negligible. For instance, when \emph{pieces} belong to the same class and so do \emph{functional constraints}, every black-box call is expected to spend a computational time that essentially depends on the number of constraints involved, i.e.\@ $|S|$. In the practical context of conic programming \cite{Nemi08}, this usually amounts to a polynomial of $d$ and $|S|$ with a leading term of order $\text{poly}(d)\cdot|S|^r$ for a suitable $r >0$,\,e.g. $r \in \{\frac{1}{2},\frac{3}{2}\}$. 
\\ \\
Naturally, one may question the possibility of avoiding the $n$ black-box calls needed to get all the $\nu^{(i)}$'s while sticking to a global optimization approach for \eqref{eq:min_problem}. In this work, our aim is to satisfy the following demand:

\begin{center}
    \textit{Can we exploit the structure in \eqref{eq:min_problem} to devise an exact algorithm that can certify (approximate) global optimality while requiring in practice much fewer than $n$ computations $\big\{\nu^{(i)}\big\}_{i\in[n]}$ required in the baseline enumeration strategy ?}
\end{center}
We answer this question positively by proposing a new algorithm dubbed \texttt{ULO} (Upper-Lower Optimization) that alternates between (a) local-search steps, which minimize upper models, and (b) suggestion steps, which optimize lower models of $F$. By carefully choosing which upper models to optimize in step (a), we can ensure the \emph{local optimality} of our final output iterate. On the other hand, within our step (b), the lower models optimizations either improve global certifications about the candidates and/or suggest a new restart for step (a). In order to illustrate our framework, we detail a first example for which we point out which functions are the \emph{pieces} and which are the \emph{constraints}. 

\example[Optimistic Robust LP] \label{example:orlp} A nominal \emph{optimistic} program (see \cite{Sinha16} for more examples) reads as follows. Suppose one can choose an option $(\beta,\gamma) \in \mathbb{R}^p \times \mathbb{R}$ among a dataset $\{(\beta^{(i)},\gamma^{(i)})\}_{i\in[n]}$ of possibilities. Then, depending on the retained option, one optimizes an affine cost $\langle \beta,u\rangle + \gamma$ over a polyhedral set of feasible actions $\mathcal{U} \subseteq \mathbb{R}^p$. Given a positive semidefinite matrix $\mathbf{B}$ and an allowance $R>0$, one also needs to satisfy a quadratic budget constraint $\langle \mathbf{B} u,u\rangle \leq R$. As such, the nominal problem, already nonconvex in general, boils down to
\begin{equation}
\min_{i \,\in\,[n]}\,\min_{u \,\in \,\mathcal{U}}\, \langle \beta^{(i)},u\rangle+\gamma^{(i)} \quad\text{s.t.} \quad \langle\mathbf{B}u,\,u\rangle \leq R.
\label{eq:nominal}
\end{equation}
One can tackle an even more challenging problem. Like in \cite{Bertsimas09} (Section 2), we take into account possible implementation errors $\Delta_u \in \mathbb{R}^p$, i.e.\@ instead of the chosen $u$ vector, $u+\Delta_u$ is implemented so that the budget constraint is eventually exceeded. Let $\bar{\alpha}>0$ describe the radius of maximal $\infty$-norm perturbation so that $\Delta_u \sim \text{Uni}(\mathbb{B}_{\infty}(\mathbf{0}_p\,;\,\bar{\alpha}))$.\\We would like the budget constraint to be satisfied in any case while optimizing the expected perturbed cost. 
At any fixed $(\beta,\gamma)$, the expected value of the objective, i.e. $\mathbb{E}_{\Delta_u}[\langle \beta, u+\Delta_u\rangle+\gamma]$, still equals $\langle \beta,u\rangle + \gamma$. The robust counterpart of \eqref{eq:nominal} becomes 
\begin{equation}\min_{i \,\in\,[n]}\,\min_{u \,\in \,\mathcal{U}}\, \langle \beta^{(i)},u\rangle+\gamma^{(i)} \quad\text{s.t.} \quad \max_{||\Delta_u||_\infty \,\leq \,\bar{\alpha}}\langle \mathbf{B}(u+\Delta_u),\,(u+\Delta_u) \rangle \leq R. \label{eq:first_ref} \end{equation}
At fixed $u \in \mathcal{U}$, \cite{Aras22} proposes a sophisticated adjustable robust optimization (ARO) mixed-integer convex reformulation of the right-hand side of \eqref{eq:first_ref} to compute the perturbation $\Delta_u^*(u)$ effectively achieving the worst-case scenario in terms of budget. \\Nevertheless, since $u$ itself is a variable, bilinear terms $\langle \mathbf{B}u,\Delta_u\rangle$ together with possible high dimensionality $p$ prevent from establishing an efficient convex reformulation in $(u,\Delta_u)$ of the robust budget constraint. To keep things tractable, one possible approximation of \eqref{eq:first_ref} is to sample $m \in \mathbb{N}$ different scenarios $\{\Delta_u^{(j)}\}_{j\in [m]}$ among the $2^p$ vertices (at least one of these achieves the worst perturbation) of $\mathbb{B}_\infty(\mathbf{0}_p\,;\,\bar{\alpha})$, and to require
$$\max_{j\,\in\,[m]}\,\langle \mathbf{B}(u+\Delta_u^{(j)}),\,(u+\Delta_u^{(j)}) \rangle =\max_{j\,\in\,[m]}\, \langle \mathbf{B} u,u\rangle + \langle 2 \mathbf{B} \Delta_u^{(j)},u\rangle + \langle \mathbf{B} \Delta_u^{(j)}, \Delta_u^{(j)}\rangle\leq R.$$ 
Note that for every $m \leq 2^p$ (equality ensured if $m=2^p$) and every $u \in \mathbb{R}^p$,$$\max_{j\,\in\,[m]}\langle \mathbf{B}(u+\Delta_u^{(j)}),\,(u+\Delta_u^{(j)}) \rangle\leq\max_{||\Delta_u||_\infty \,\leq \,\bar{\alpha}}\langle \mathbf{B}(u+\Delta_u),\,(u+\Delta_u) \rangle.$$
Hence, the following \eqref{eq:min_problem} problem with the extra variable $\eta \geq \langle \mathbf{B}u,u\rangle $,
\begin{equation}
\min_{u \,\in\,\mathcal{U},\,\langle \mathbf{B}u,u\rangle \leq \eta}\,\min_{i \,\in\, [n]}\,\langle \beta^{(i)},u\rangle+\gamma^{(i)}\hspace{2pt}:\hspace{2pt}\max_{j \in [m]} \eta + \langle 2 \mathbf{B} \Delta_u^{(j)},u\rangle + \langle \mathbf{B} \Delta_u^{(j)}, \Delta_u^{(j)}\rangle -R \leq 0.
\label{eq:example1}
\end{equation}
%\vspace{-20pt}
\begin{table}[ht] 
\centering
\renewcommand*{\arraystretch}{1.23}
\begin{tabular}{|c||c|c|}
\hline
   \textbf{elements} & expression & \textit{typical} structure \\
    \hline
    $d$ & $p+1$ & $10^2 \to 10^3$ \\
    \hline
    $\mathcal{X}$ & $\{(u,\eta) \in \mathcal{U} \times \mathbb{R}\,|\,\langle \mathbf{B}u,u\rangle \leq \eta\}$ & polyhedron $\cap$ second-order cone\\
    \hline
    $f^{(i)}$ & $\langle \beta^{(i)},u\rangle+\gamma^{(i)}$ &  \multirow{2}{*}{linear in $u$}\\
    $c^{(j)}$ & $ 
 \eta + \langle 2 \mathbf{B} \Delta_u^{(j)},u\rangle +\langle \mathbf{B}\Delta_u^{(j)},\Delta_u^{(j)}\rangle -R$ & \\
    \hline
\end{tabular}
\vspace{5pt}
\caption{Optimistic Robust LP translated as an instance of \eqref{eq:min_problem}.}
\label{tab:orlp}
\end{table}

\subsection{Related work}
\label{rel_work}
The class \eqref{eq:min_problem} includes many problems of interest. For example, 
in \cite{Pansari17} authors study the minimization of a sum of \emph{$\lambda$-truncated} (or \emph{clipped} \cite{Liu18}) convex functions like 
\begin{equation}
F^* = \min_{x \in \mathbb{R}^d}\, \frac{1}{N}\,\sum_{s=1}^{N}\,\min\{h_1^{(s)}(x),\lambda^{(s)}\}
\label{eq:truncated_cvx}
\end{equation}
where for each $s\in [N]$, $h_1^{(s)} : \mathbb{R}^d \to \mathbb{R}$ represents a continuous convex \emph{component loss} whose value can not exceed a fixed threshold $\lambda^{(s)} \in \mathbb{R}$. We can cast  \eqref{eq:truncated_cvx} as an instance of \eqref{eq:min_problem} where the set of \emph{functional constraints} is empty ($m=0$) and the basic set $\mathcal{X} = \mathbb{R}^d$. Since there exists $n=2^N$ ways to select a subset $I \subseteq [N]$, one can link every $i \in \{1,\dots,2^N\}$ to a single $I^{(i)}\subseteq[N]$, itself assigned to a convex \emph{piece} $f^{(i)}$ defined as 
$$x \to f^{(i)}(x) =  \frac{1}{N}\,\Bigg(\sum_{s \in I^{(i)}}\,h_1^{(s)}(x) + \sum_{s \in [N]\backslash I^{(i)}}\,\lambda^{(s)}\Bigg).$$
More recently, \cite{Zuo23} extends \eqref{eq:truncated_cvx} by replacing the constant $\lambda^{(s)}$ by convex functions $h_2^{(s)}$.\\ Actually, any problem that fits the structure of sum of minimum-of-convex (SMC) programming \eqref{eq:smc_ex}, can be seen as a \eqref{eq:min_problem}. (SMC) reads as follows. For every $s\in[N]$, one considers $n_s \in \mathbb{N}$ continuous convex functions $h^{(s)}_l$ for every $l \in [n_s]$. Thenceforth, a possible extension of \eqref{eq:truncated_cvx} stems as 
\begin{equation}
F^* = \min_{x \in \mathbb{R}^d}\, \frac{1}{N}\,\sum_{s=1}^{N}\,\min_{l\,\in\,[n_s]} \,h_l^{(s)}(x).
\label{eq:smc_ex}
\end{equation}
Notably, generalized clustering \cite{Ding24} with $n_s = B$ for every $s\in [N]$, $x=(x^{(1)},\dots,x^{(B)})$, $h^{(s)}_l(x) = h(x^{(l)},\bar{x}^{(s)})$ and $h(\cdot,\bar{x}^{(s)})$  continuous convex (e.g. $h(x,\bar{x}^{(s)})=||x-\bar{x}^{(s)}||_2^2$), falls into this (SMC) category. One proceeds to the \eqref{eq:min_problem} reformulation of \eqref{eq:smc_ex} by simply choosing that every
\emph{piece} $f^{(i)}$ is constructed based upon an unique selection $\sigma=(\sigma_1,\dots,\sigma_N) \in \bigtimes_{s=1}^{N}\,[n_s]$ of $h^{(s)}_l$ terms as in 
\begin{equation}
x \to f^{(i)}(x) =  \frac{1}{N}\,\sum_{s=1}^N\,h^{(s)}_{\sigma_s}(x).
\label{eq:SMC}
\end{equation}
    About \eqref{eq:SMC}, one counts $n = \Pi_{s=1}^{N}\,n_s$ different $\sigma$-selections and thereby as many \emph{pieces} regarding our \eqref{eq:min_problem} framework, generalizing the case \eqref{eq:truncated_cvx} for which $n_s=2$ for every $s \in[N]$ and thus $n=2^N$. Again, the number $n$ grows exponentially with $N$. Then, solving (SMC) using the baseline \emph{enumeration} strategy \eqref{eq:enumeration_scheme_ref} amounts to solve a finite yet exponential number of convex problems, out of scope for a large $N$. To the best of our knowledge, a single paper \cite{Liu18} managed to propose an exact algorithm to solve \eqref{eq:truncated_cvx} while only computing the optimal value of a polynomial number of convex problems. \\In some quite specific situations and in a low dimensional setting ($d \in \{1,2\}$), \cite{Liu18} shows that it is possible to efficiently segment the basic feasible set $\mathcal{X}$ in $\mathcal{O}(N^d)$ distinct regions; each of which one is aware of the unique \emph{piece} $f^{(i^*)}$ achieving the minimum among the $n$ \emph{pieces}. Hence, scanning through all these $\mathcal{O}(N^d)$ regions one at the time and optimizing its associated active \emph{piece} will solve \eqref{eq:truncated_cvx} with certainty. However, when $d>2$, the exponential complexity of an \emph{enumeration} strategy when tackling \eqref{eq:truncated_cvx} and, by extension, \eqref{eq:SMC}, seems unavoidable. Actually, as noted by \cite{Barratt20}, a particular instance of \eqref{eq:truncated_cvx} resides in the subset sum problem (SSP), proven to be NP-complete. \\
    
    \noindent Thereby, local-search heuristics are employed to tackle \eqref{eq:truncated_cvx} in \cite{Sinha16,Barratt20} without global guarantees in the general case. That being stated, under mild additional assumptions, authors in \cite{Barratt20} propose an equivalent mixed-integer conic programming (MICP) reformulation of \eqref{eq:truncated_cvx} based on perspective functions. By relaxing its integrality constraints, one can then compute a lower bound $\check{F} \leq F^*$. The optimal solution of the relaxed reformulation serves as starting point to the aforementioned heuristics. These latter return a candidate solution associated with an objective value $\hat{F} \geq F^*$. 
    In such circumstances, this candidate is certified with $\hat{F}-\check{F}$ absolute global optimality. \\ \vspace{-5pt}\\ Although (SMC) falls in the scope of \eqref{eq:min_problem}, our algorithm \texttt{ULO} is not the method of choice for a very large number $n$ of \emph{pieces}, as it is often the case in \eqref{eq:truncated_cvx} for machine-learning experiments for which lots of \emph{data} appears as a desirable feature. As previously seen with the clustering example, $N$ matches the number of data points hence $n$ would scale exponentially as $B^N$. On the contrary, we are focused on globality results. Unlike previously cited works, our main attention is driven by \emph{functional constraints} that render the evaluation of $\nu^{(i)}$ for every piece $i \in [n]$ costful, like in \eqref{eq:example1}. By computing values $\nu(\{i\},S)$ as in \eqref{eq:small_min_problem_i} for small-sized $S\subseteq [m]$, we produce lower bounds $\check{F}$ at reduced time expenses, eventually leading to a satisfying optimality certificate.  

\subsection{Contributions} 
 We summarize our main contributions as follows.
 \begin{itemize} % although focus on global: local conditions of independent interest 
     \item First, we give a valuable insight (Section \ref{sec:insight}) into \eqref{eq:min_problem} together with meaningful examples (Section \ref{sec:motivation}). Furthermore, we provide sufficient local optimality conditions (Section \ref{sec:opt_cond}), tailored for \eqref{eq:min_problem}. We present (Section \ref{sec:rationale}) and prove the exactness (Section \ref{sec:correctness}) of our new algorithm \texttt{ULO}, using the above conditions to elect our output iterate among visited candidates. That is, if \texttt{ULO} is stopped before exact global optimality, it will return a local minimum of \eqref{eq:min_problem}, usually hard to certify in a nonconvex nonsmooth setting. \\
     \item Second, assuming a time complexity model when using our black-box \eqref{eq:small_min_problem_i}, we conduct a theoretical study (Section \ref{subsec:dag}) of the overall computational cost of both the baseline \texttt{Enumeration Scheme} \eqref{eq:enumeration_scheme_ref}, and our algorithm \texttt{ULO}. \\To that end, we develop a stochastic directed acyclic graph (DAG) abstraction of \eqref{eq:min_problem} instances whose difficulty relies on controllable parameters. By simulating \texttt{ULO} many times on those graph abstract problems, we identify regimes, defined by the control parameters, for which \texttt{ULO} is expected to outperform the baseline.\\
     \item Third, we demonstrate the validity of our approach with some numerical experiments on practical tasks (Section \ref{num_experiments}). We compare (emphasis put on effective wall-clock time) \texttt{ULO} against a pure \texttt{Enumeration Scheme} but also against a \texttt{Restarted Alternating Minimization} algorithm, based on an extension of the bi-convex formulation in \cite{Barratt20} (Section 3. \textsection \,Alternating minimization).

 \end{itemize}

\subsection{Preliminaries}
\label{prel} 

For any $d\in \mathbb{N}$,
we endow $\mathbb{R}^{d}$ with the usual dot product and the classical Euclidean norm. 
For any $p$-norm $||\cdot||_p$ ($p \geq 1$) on $\mathbb{R}^d$, $R \geq 0$ and $\bar{x} \in \mathbb{R}^d$,
\begin{equation}
\mathbb{B}_{p}(\bar{x}\,;\,R) := \big\{x \in \mathbb{R}^d\,|\, ||x-\bar{x}||_p \leq R \big\}.
\end{equation}

\noindent The handy notation $\mathbf{1}_d$ (respectively $\mathbf{0}_d$) represents the \emph{all-ones} (respectively \emph{all-zeroes}) vector of size $d$. However, when the context is clear, we will drop the subscript. \\
Let $h : \mathbb{R}^d \to \mathbb{R} \cup \{\infty\}$, the (natural) \emph{domain} of $h$ noted $\text{dom} \,h$ is defined as 
\begin{equation}
\text{dom} \,h := \{x\in\mathbb{R}^d\,|\,h(x)<\infty\}.
\label{def:domain}
\end{equation}
Let $N \in \mathbb{N}$. The $N$-\emph{standard simplex} reads
\begin{equation}
\Delta^N := \{q \in \mathbb{R}^N\,|\,\langle\mathbf{1}_N,q\rangle = 1,\, q\geq \mathbf{0}_N\}.
\end{equation}
The \emph{convex hull} of a set $X \subseteq \mathbb{R}^d$ written $\textbf{conv}(X)$ represents the smallest convex set containing $X$. This set is described as
\begin{equation}
\textbf{conv}\big(X\big) := \Bigg\{\sum_{s=1}^{\hat{N}}\,q_s \cdot g^{(s)}\,\bigg|\,\hat{N} \in \mathbb{N},\,g^{(s)} \in X\hspace{3pt}\forall s \in [\hat{N}] \wedge q \in \Delta^{\hat{N}}\Bigg\}. 
\label{eq:cxv_hull_gen}
\end{equation}
If it turns out that $X=\{g^{(s)}\}_{s=1}^N$ then this set explicitly boils down to
\begin{equation}
\textbf{conv}\big(X\big) = \Bigg\{\sum_{s=1}^{N}\,q_s\cdot g^{(s)}\,\bigg|\,q \in \Delta^N\Bigg\}.
\label{eq:cxv_hull}
\end{equation}
Let $h : \mathbb{R}^d \to \mathbb{R}$ be locally Lipschitz continuous (differentiable almost everywhere \cite{NSNC20}), its \emph{Clarke subdifferential set} at any $x \in \mathbb{R}^d$ is well-defined and corresponds to
\begin{equation}
\partial_C h(x) = \textbf{conv}(D_h(x)), \quad D_h(x) :=\Bigg\{ g \in \mathbb{R}^d \,:\begin{array}{l}
\exists \,x_k \to x,\,\nabla h(x_k) \text{ exists}\\
\nabla h(x_k) \to g
\end{array}
\Bigg\}.
\label{eq:clarke}
\end{equation}
\remark \label{rem:nondiff}If $\partial_C h(x)$ is not a singleton then $h$ is not differentiable at $x$.

\begin{definition}[Local optimality] \label{def:locopt} A point $\hat{x} \in \mathcal{X}$ is called a \emph{local minimum} of \eqref{eq:min_problem} if $F(\hat{x})<\infty$, i.e. $\hat{x} \in \text{dom}\,F$, and there exists a radius $\alpha>0$ such that 
\begin{equation}
 F(\hat{x}) = \min_{x\, \in \,\mathcal{X} \,\cap \,\mathbb{B}_2(\hat{x}\,;\,\alpha)}\,F(x).
\end{equation}
\end{definition}
\noindent Let $x\in \mathcal{X}$. We will often need to collect all the indices of \emph{pieces} that achieve the minimal value among all the \emph{pieces}, i.e. $\min_{i\,\in\,[n]}\,f^{(i)}(x)$. Since numerical calculations are prone to (small) computational errors, we set a tolerance $\rho\geq0$ and look after the \emph{pieces} whose value at $x$ does not exceed $\min_{i\,\in\,[n]}\,f^{(i)}(x)+\rho$, as suggested below.
\begin{definition}[Active sets] \label{def:active_set} Let $\rho\geq0$. The ($\rho$ -)\emph{active set} at $x \in \mathcal{X}$ is the subset 
\begin{equation}
\mathcal{A}_{\rho}(x) := \Big\{ i^*\in [n]\,\big|\, f^{(i^*)}(x) \leq \min_{i\,\in [n]\,}\,f^{(i)}(x)+\rho\Big\}.
\label{eq:as_min}
\end{equation}
\end{definition}
\noindent By convention, $\mathcal{A}(x) := \mathcal{A}_{0}(x)$, simply called \emph{active set} at $x$.\\

\noindent When \emph{pieces} are continuous (which we assume), an important feature of \emph{active sets} resides in their nonexpansiveness in the direct neighborhood of a point $\hat{x}$, as highlighted by the next proposition. 
\begin{proposition} 
\label{local_non_expansiveness}
Let $\hat{x}\in\mathcal{X}$. There exists  $\alpha >0$ such that for every $x \in \mathcal{X} \cap \,\mathbb{B}_2(\hat{x}\,;\,\alpha)$, $$\mathcal{A}(x) \subseteq \mathcal{A}_\rho(\hat{x}) \quad\forall \rho\geq0.$$
\end{proposition}
\begin{proof}
Let $\bar{h} = \min_{i \in \mathcal{A}(\hat{x})}\,h^{(i)}$, $\bar{h}$ is continuous since $h^{(i)}$ is continuous for every $i\,\in\, [n]$. By definition, for every $\tilde{i}\in [n]\backslash \mathcal{A}(\hat{x})$, there must exist a slack $\delta(\tilde{i})>0$ such that $$h^{(\tilde{i})}(\hat{x}) = \bar{h}(\hat{x})+\delta(\tilde{i}).$$
Noticing that the function $\varrho^{(\tilde{i})} := h^{(\tilde{i})}-\bar{h}$ must also be continuous, it comes that there exists a strictly positive radius $R(\tilde{i})>0$ such that for every $x \in \mathcal{X}\,\cap\,\mathbb{B}_2(\hat{x}\,;\,R(\tilde{i}))$, $$\varrho^{(\tilde{i})}(x) > 0$$since $\varrho^{(\tilde{i})}(\hat{x}) = \delta(\tilde{i})>0$ by construction. 
Then, taking $\alpha = \min_{\tilde{i}\,\in \,([n]\backslash \mathcal{A}(\hat{x}))}\,R(\tilde{i})$ leads to a neighborhood of $\hat{x}$, i.e. $\mathcal{X} \,\cap \,\mathbb{B}_2(\hat{x}\,;\,\alpha)$, on which $\varrho^{(\tilde{i})}$ functions are all strictly positive. Hence, no index $\tilde{i} \in [n]\backslash \mathcal{A}(\hat{x})$ is such that $h^{(\tilde{i})}(x) < \bar{h}(x)$ for any $x$ of that neighborhood. It comes then that for every $x \in \mathcal{X} \,\cap \,\mathbb{B}_2(\hat{x}\,;\,\alpha)$, $$ \mathcal{A}(x) \subseteq \mathcal{A}(\hat{x}).$$
Finally, for any $\rho\geq0$, it always holds $\mathcal{A}(\hat{x})=\mathcal{A}_0(\hat{x}) \subseteq\bar{\mathcal{A}_\rho}(\hat{x})$ by definition.  \end{proof}
\noindent In other words, there exists a ball around $\hat{x}$ over which no other \emph{piece} than those already \emph{active} at $\hat{x}$ (i.e. contained in $\mathcal{A}(\hat{x})$) become \emph{active}. Subsequently, the analytical expression of $F$ on such ball relies on the (probably reduced number of) \emph{pieces} in $\mathcal{A}(\hat{x})$.\\We will formally lay down and use this fact in the proof of Lemma \ref{lemma:lo}.
%\begin{equation}
%F(x) = \min_{i\,\in\,\mathcal{A}(x)}\,f^{(i)}(x) = \min_{i\,\in\,\mathcal{A}_\rho(\hat{x})}\,f^{(i)}(x)\quad \forall x \,\in\,\mathcal{X}\,\cap\,\mathbb{B}_2(\hat{x}\,;\,\alpha)\,\cap\,\mathcal{D}^{([m])}.
%\label{eq:minimalist_rep}
%\end{equation}

\begin{assumption}

We make the following assumptions. Let $(H,S) \in [n] \times [m]$.

\begin{itemize}
\item[A.] One has access to an oracle computing  the value 
\begin{equation}
\nu(H,S) = \min_{x \in \mathcal{X}}\,\min_{i\, \in \,H}\,f^{(i)}(x) \quad \text{s.t.}\quad \max_{j \in S}\,c^{(j)}(x) \leq 0
\label{eq:oc1}.
\end{equation}

\item[B.] In addition, we further assume the existence and the availability of a (finite) minimizer $x^*(H,S) \in \mathcal{X} \cap \mathcal{D}^{(S)}$ such that,
\begin{equation}
\min_{i \in H}\,f^{(i)}(x^*(H,S)) = \nu(H,S).
\label{eq:oc2}
\end{equation}
Moreover, the oracle deterministically outputs $x^*(H,S)$. That is, it will always return the same $x^*(H,S)$ when provided with $(H,S)$ as arguments.\\

\item[C.] For every $i\in [n]$, $f^{(i)}$ is continuous and convex on its domain with 
\begin{equation} \mathcal{X} \subseteq \text{dom}\,f^{(i)}.
\label{eq:inclusion}
\end{equation}
\end{itemize}
\label{A1}
\end{assumption}

\section{Minimum structured optimization}
\label{models}
We start by discussing what causes \eqref{eq:min_problem} problems to be hard. We emphasize that even when \emph{pieces} $f^{(i)}$ and \emph{constraints} $c^{(j)}$ are smooth convex, our objective may be nevertheless nonconvex and nonsmooth in general. However, in this latter case, the source of nonconvexity is entirely explained and understood through the \texttt{min} operator.  \\
We recall at this stage that the objective $F$ amounts to $F^{([n],[m])}$ as expressed in \eqref{eq:min_problem_equiv}.
\subsection{Main features}
\label{sec:insight}
According to \eqref{eq:small_F} and \eqref{eq:min_problem_equiv}, the (natural) \emph{domain} of $F$ comes as follows \begin{equation}
\text{dom}\,F = \bigcup_{i=1}^{n}\,\Big(\mathcal{D}^{([m])} \,\cap\,\text{dom}\,f^{(i)}\Big) = \mathcal{D}^{([m])}\,\cap\,\bigcup_{i=1}^{n}\,\text{dom}\,f^{(i)}.
\label{eq:domain_F_equiv}
\end{equation} It represents the set of points that satisfy all the \emph{functional constraints} of \eqref{eq:min_problem} and for which at a least one \emph{piece} takes a finite value. When every \emph{piece} is continuous over the whole ambient space $\mathbb{R}^d$, we observe that $\text{dom}\,F = \mathcal{D}^{([m])}$. Looking back to Assumption \ref{A1} (C), from the fact that $\mathcal{X}\subseteq \text{dom}\,f^{(i)}$ for every $i\in[n]$, we deduce that 
\begin{equation}
F(x) < \infty \quad \forall x \in \mathcal{X}\,\cap\,\mathcal{D}^{([m])}.
\label{eq:overall_domain}
\end{equation}
We notice that \eqref{eq:min_problem} can also be rewritten compactly as 
\begin{equation}
F^* =\min_{x \,\in \, \mathcal{X}\,\cap\,\mathcal{D}^{([m])}} \,\min_{i \in [n]} f^{(i)}(x).
\label{eq:min_problem_compact}
\end{equation}
Equation \eqref{eq:min_problem_compact} underlines that our problem is nothing else than a constrained optimization problem for which the feasible set $\mathcal{X}\,\cap\,\mathcal{D}^{([m])}$ might be convex but for which the objective $\min_{i\in[n]}\,f^{(i)}$ is usually not smooth nor convex. Even in the simple case where \emph{pieces} and \emph{constraints} are all affine, $\mathcal{X}\,\cap\,\mathcal{D}^{([m])}$ is convex if $\mathcal{X}$ is convex (e.g.\@ a polyhedron if $\mathcal{X}$ is a polyhedron) but $\min_{i\in[n]}\,f^{(i)}$ would be piecewise-linear concave.\\

\noindent Without loss of generality, there exists a partition of the feasible set of \eqref{eq:min_problem_compact}  into $Y \in \mathbb{N}$ distinct regions $\{\mathcal{R}^{(e)}\}_{e \in [Y]}$ such that for every $e \in [Y]$, there is $i^*(e)\in[n]$ with  \begin{equation}F(x)= \min_{i \,\in\,[n]}\,f^{(i)}(x) = f^{(i^*(e))}(x) \quad \forall x \in \mathcal{R}^{(e)}.\label{eq:determination} \end{equation} \paragraph*{Nonsmoothness.}
Every $x \in \text{dom}\,F$ for which $|\mathcal{A}(x)|>1$ represents a point where the objective is not differentiable unless active \emph{pieces}' subdifferential sets boil down to a common singleton. When \emph{pieces} are smooth, the \emph{Clarke subdifferential set} reads \cite{NSNC20}
\begin{equation}
\partial_C F(x) = \textbf{conv}\big(\{\nabla f^{(i)}(x)\,|\,i\in\mathcal{A}(x)\}\big).
\label{eq:clarke_subdiff}
\end{equation}
If $\nabla f^{(\star)}(x)\not=\nabla f^{(\times)}(x)$ for $\star\not=\times$ indices in $\mathcal{A}(x)$ then $\partial_C F(x)$ contains more than one element thus implying the non-differentability of $F$ at $x$ (see Remark \ref{rem:nondiff}). \paragraph*{Nonconvexity.} Let $\nu\in \mathbb{R}$ and let $I \subseteq [n]$, the $\nu$-sublevel set of $\min_{i \,\in\,I}\,f^{(i)}$ reads $$ \mathcal{S}_{I,\nu} = \bigcup_{i\,\in\,I}\,\big\{x \in \mathbb{R}^d\,|\, f^{(i)}(x) \leq \nu\big\}.$$ 
Thereby, the term $\min_{i \,\in\,[n]}\,f^{(i)}$ fails to be convex as soon as $\mathcal{S}_{[n],\nu}$ is not convex for a fixed value of $\nu$. This happens when there exists $\tilde{i}\in[n]$ such that $\mathcal{S}_{\{\tilde{i}\},\nu} \backslash \mathcal{S}_{[n]\backslash\{\tilde{i}\},\nu}$ is not convex. From a geometrical point of view, this would imply that there are points $x_\star \in \mathcal{S}_{\{\tilde{i}\},\nu}$ and $x_\times \in \mathcal{S}_{[n]\backslash\{\tilde{i}\},\nu}$ such that the segment $[x_\star,x_\times]$ is not contained in $\mathcal{S}_{[n],\nu}$.\\

\noindent In order to clarify the notations covered so far, we illustrate on Figure \ref{fig:MS_intro} a toy example, relatively general case of \eqref{eq:min_problem}. Its global optimum is located at $x^* = \sqrt{3/2}$.
\begin{equation}
\begin{aligned}
\min_{x \,\in \,\mathbb{R}} \quad &  \min \{10^{-1}(x-1)^2,\,\max\{x-2,-2x\}-1,\,1+\exp(x/5),\,10(x+4)^2\}
\label{eq:illustration_1D}\\
\textrm{s.t.} \quad & \max \{(3/2)-x^2,\,x-5,\,-x-5\} \leq 0  \\ 
\end{aligned}
\end{equation}
One identifies $\mathcal{X} = \mathbb{R}$,\,$n=4$ (every \emph{piece} is convex),\,$m=3$ (the two last \emph{functional constraints} are convex but not the first one). Overall, $\text{dom}\,F = [-5,-\sqrt{3/2}] \cup [\sqrt{3/2},5]$.
\begin{figure}[h]
\vspace{-5pt}
\centering
\includegraphics[width=0.8\textwidth]{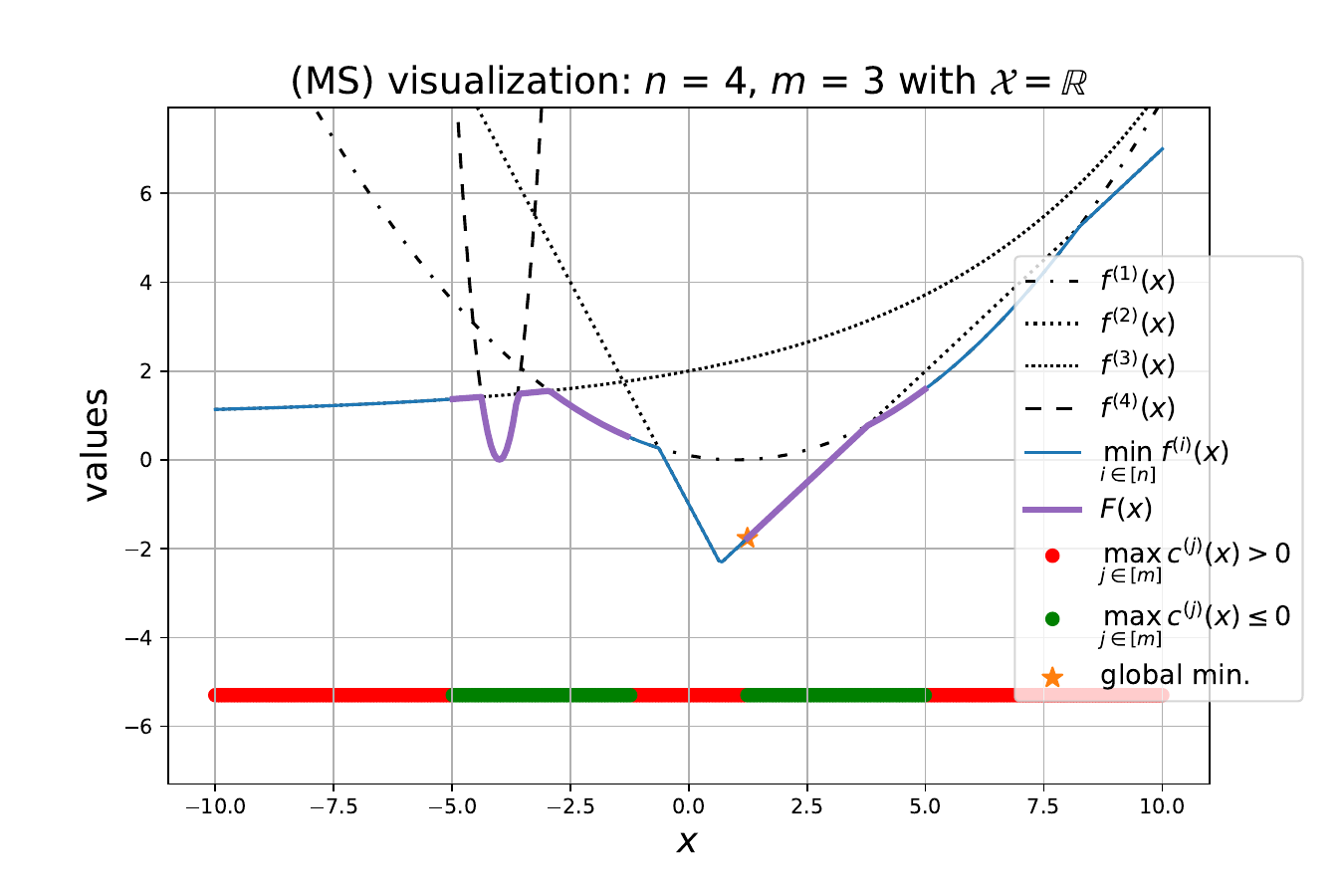}
\vspace{-10pt}
\caption{Graph and illustration of problem \eqref{eq:illustration_1D}.}
\label{fig:MS_intro}
\end{figure}\\
Let us now detail two more examples that motivate the framework of \eqref{eq:min_problem} and the further use of our algorithm \texttt{ULO} as an efficient tool to tackle them. 
\subsection{Further examples}
\label{sec:motivation}
Every piecewise-linear objective can be rewritten as a difference between two convex piecewise affine functions \cite{Schulze87}. Among others, these functions appear in recent works \cite{plmr20,Ho21,Bagirov22} to fit models as cheap surrogates to complicated and potentially nonconvex general objectives. 
 It is straightforward to reformulate the optimization of piecewise-linear objectives as instances of \eqref{eq:min_problem} as shown in the next example, very similar to our introductory Example \ref{example:orlp}.
    \example[Piecewise-Linear objectives] \label{example:pl} Let $\{(\beta_1^{(j)},\gamma_1^{(j)})\}_{j \in [m]}$, $\{(\beta_2^{(i)},\gamma_2^{(i)})\}_{i \in [n]}$ be two collections of tuples from $\mathbb{R}^p \times \mathbb{R}$ defining a difference-of-convex (piecewise-linear) function $f : \mathbb{R}^p \to \mathbb{R}$,
\begin{equation}
    u \to f(u) = \Bigg[\max_{j \in [m]}\,\langle \beta_1^{(j)},u\rangle+\gamma_1^{(j)}\Bigg]-\Bigg[\max_{i \in [n]}\,\langle \beta_2^{(i)},u\rangle+\gamma_2^{(i)}\Bigg].
\label{eq:pl_def}
\end{equation}
Using a standard epigraph technique, minimizing $f$ on any polyhedral set $\mathcal{U}\subseteq \mathbb{R}^p$ can be equivalently stated as 
\begin{equation}
\begin{aligned}
\min_{u \,\in\,\mathcal{U}}\,f(u) \equiv \min_{(u,\eta) \,\in \,\mathcal{U}\times\mathbb{R}} \quad &  \min_{i\in[n]}\, \eta -\gamma_2^{(i)}-\langle \beta_2^{(i)},u\rangle
\label{eq:example_pl_opt}\\
\textrm{s.t.} \quad & \max_{j\in [m]}\, -\eta+\gamma_1^{(j)}+\langle \beta_1^{(j)},u\rangle \leq 0. \\ 
\end{aligned}
\end{equation}
\vspace{-10pt}
\begin{table}[ht] 
\centering
\renewcommand*{\arraystretch}{1.23}
\begin{tabular}{|c||c|c|}
\hline
   \textbf{elements} & expression & \textit{typical} value \\
    \hline
    $d$ & $p+1$ & $10^2 \to 10^4$ \\
    \hline
    $\mathcal{X}$ & $\mathcal{U} \times \mathbb{R}$ & polyhedron \\
    \hline
    $f^{(i)}$ & $\eta-\gamma_2^{(i)}-\langle \beta_2^{(i)},u\rangle$ &  \multirow{2}{*}{linear in $u$}\\
    $c^{(j)}$ & $-\eta+\gamma_1^{(j)}+\langle \beta_1^{(j)},u\rangle $ & \\
    \hline
\end{tabular}
\vspace{7pt}
\caption{Piecewise-linear program translated as an instance of \eqref{eq:min_problem}.}
\label{tab:pl_prog}
\end{table}

A different example is the \emph{continuous relaxation} of difficult \emph{disjunctive programming} problems \cite{Balas18}. That is, one lifts and strongly penalizes the least violated constraint among the set of disjunctive constraints as explained in the following example.

\example[Relaxed Disjunctive Programming] \label{example:disprog} Let $(\mathbf{B},g) \in \mathbb{R}^{m \times p} \times \mathbb{R}^m$ be a data matrix together with linked noisy measurements. The typical $\ell_\infty$ regression task asks to recover coefficients $\bar{u} \in \mathbb{R}^p$ such that $\mathbf{B} \bar{u}\simeq g$ in the sense of minimal worst deviation, i.e.\@ with the smallest residual $||\mathbf{B}\bar{u}-g||_\infty$. In certain circumstances, one knows prior results based on previous regressions with large amounts of data. That is, thanks to the results of $n \in \mathbb{N}$ previous fits, one obtained a collection $\{u^{(i)}\}_{i \in [n]}$ of $n$ \emph{a-priori} plausible values for $\bar{u}$. The goal is then to update one of these previous models at the light of the new data $(\mathbf{B},g)$ at hand. More precisely, one looks for $\bar{u}$ that best explains the aforementioned new data while keeping $\bar{u}$ at a maximal distance of $R>0$ from at least one of the previous models $u^{(i)}$. Hence, the disjunctive model 
\begin{equation}
\begin{aligned}
\min_{u \,\in \,\mathbb{R}^p} \quad &  ||\mathbf{B}u-g||_\infty
\label{eq:example_rdp}\\
\textrm{s.t.} \quad & u \in \bigcup_{i\in[n]}\,\mathbb{B}_2(u^{(i)}\,;\,R) \\ 
\end{aligned}
\end{equation}
A possible relaxation of the above allows to search for $u$ outside the feasible set of \eqref{eq:example_rdp} but penalizes the excess distance from $u$ to the closest \emph{a-priori} explanation $u^{(i)}$, i.e.$$\min_{i\,\in\,[n]}\,\max\{||u-u^{(i)}||_2-R,0\} = \begin{cases} 0 & \exists\,i \,\in\,[n],\,||u-u^{(i)}||_2\leq R\\\min_{i\,\in\,[n]}\,||u-u^{(i)}||_2-R&\text{otherwise}.\end{cases}$$ \noindent To strengthen the relaxation, we multiply the penalty by a large constant $C_{\text{pen}} >0$. \\For any $j \in [m]$, let $(\beta^{(j)},\gamma^{(j)})$ denote $(\mathbf{B},g)$'s $j$-th observation, the problem becomes 
\begin{equation}
\begin{aligned}
\min_{(u,\eta) \,\in \,\mathbb{R}^p \times \mathbb{R}} \quad & \min_{i\in[n]}\, \eta + C_{\text{pen}}\cdot\max\{||u^{(i)}-u||_2-R,0\}
\label{eq:example_rdp_new}\\
\textrm{s.t.} \quad & \max_{j \in [m]}\,-\eta + |\langle \beta^{(j)},u\rangle-\gamma^{(j)}|\leq 0 \\ 
\end{aligned}
\end{equation}
\vspace{-10pt}
\begin{table}[ht] 
\centering
\renewcommand*{\arraystretch}{1.23}
\begin{tabular}{|c||c|c|}
\hline
   \textbf{elements} & expression & \textit{typical} value \\
    \hline
    $d$ & $p+1$ & $ 10^2 \to 10^4$ \\
    \hline
    $\mathcal{X}$ & $\mathbb{R}^p \times \mathbb{R}$ & whole domain \\
    \hline
    $f^{(i)}$ & $\eta + C_{\text{pen}}\cdot\max\{||u^{(i)}-u||_2-R,0\}$ &  convex \\
    %\hline
    $c^{(j)}$ & $-\eta + |\langle \beta^{(j)},u\rangle-\gamma^{(j)}|$& polyhedral convex\\
    \hline
\end{tabular}
\vspace{5pt}
\caption{Relaxed disjunctive program translated as an instance of \eqref{eq:min_problem}.}
\label{tab:rdp}
\end{table}
\vspace{-15pt}
\subsection{Upper and lower models} 
We derive now global models of the objective $F$ that ultimately lead to subproblems of \eqref{eq:min_problem}, either yielding upper bounds $\hat{F}$ or lower bounds $\check{F}$ on the optimal value $F^*$. These models are not necessarily tight in the sense that there might not exist a point in the set $\mathcal{X} \cap \mathcal{D}^{([m])}$ such that the value of the models and the value of $F$ coincide. 

\paragraph*{Upper models} Let $H\subseteq [n]$. For both $F^{(H,[m])}$ and $F=F^{([n],[m])}$, all the \emph{constraints} are included in the model. In general, $\text{dom}\,F \supseteq \text{dom}\,F^{(H,[m])}$ but equality holds, for example, if all the \emph{pieces} in $H$ are full-domain. Nevertheless, at fixed $x \in \mathcal{D}^{([m])}$, one only chooses a minimal value among \emph{pieces} in $H \subseteq [n]$ for $F^{(H,[m])}$, hence the inequality 
\begin{equation}
 F^{(H,[m])}(x) \geq F(x) \quad \forall x \in \text{dom}\,F^{(H,[m])}.
 \label{eq:upper_model}
\end{equation}
If $\mathcal{A}(x) \cap H = \emptyset$ for every $x \in \mathcal{X} \cap \mathcal{D}{([m])}$, $F^{(H,[m])}$ is not tight since \emph{pieces} from $H$ are nowhere \emph{active} on the feasible set.
Minimizing both sides of \eqref{eq:upper_model} on $\mathcal{X}$ yields
\begin{equation}
    \hat{F}(H) := \min_{x\in \mathcal{X}}\,F^{(H,[m])}(x) =\min_{i \,\in\, H}\,\nu^{(i)} \geq F^*.
\label{eq:upper_bound_gen}
\end{equation}
\paragraph*{Lower models} Symmetrically to \eqref{eq:upper_model} where all the \emph{constraints} are considered in $S$ in the expression $F^{(H,S)}$, here one keeps all the \emph{pieces} in $H$. For any $S\subseteq [m]$, 
\begin{equation}
 F^{([n],S)}(x) \leq F(x) \quad\forall x \in \text{dom}\,F^{([n],S)}.
 \label{eq:lower_model}
\end{equation}
When $x \in \mathcal{D}^{([m])}$, $F^{([n],S)}$ matches $F$ in the sense that either $F(x)=F^{([n],S)}(x)=\infty$ if $x\not \in \cup_{i \in [n]}\,\text{dom}\,f^{(i)}$ (see \eqref{eq:domain_F_equiv}) or $F(x)=F^{([n],S)}(x)<\infty$ otherwise. On the contrary, if $x \in \text{dom}\,F^{([n],S)} \cap (\mathbb{R}^d \backslash \mathcal{D}^{([m])})$, $F(x)=\infty$ whereas $F^{([n],S)}$ takes a finite value. One concludes that $\text{dom}\,F^{([n],S)}\supseteq \text{dom}\,F$.
Again, minimizing both sides of \eqref{eq:lower_model} on $\mathcal{X}$ yields 
\begin{equation}
    \check{F}(S) := \min_{x\in \mathcal{X}}\,F^{([n],S)}(x) \leq F^*.
\label{eq:lower_bound_gen}
\end{equation}
\subsection{Optimality conditions}\label{sec:opt_cond} We end this section by proving a lemma providing a convenient sufficient condition \eqref{eq:scan} to deduce \emph{local optimality} as understood in Definition \ref{def:locopt}.
\begin{lemma}
\label{lemma:lo}
Let $i^*\in [n]$, $\rho\geq0$ and $x^* \in \arg\min_{x\,\in\,\mathcal{X}}\,F^{(\{i^*\},[m])}(x)$. If it holds
\begin{equation} \min_{x \,\in\,\mathcal{X}}\,F^{(\{i\},[m])}(x)= \nu^{(i)} \geq \nu^{(i^*)} = F^{(\{i^*\},[m])}(x^*) \quad \forall i \in \mathcal{A}_\rho(x^*)
\label{eq:scan}
\end{equation}
then $x^*$ is a local optimum of \eqref{eq:min_problem}.
\end{lemma}
\begin{proof} Let $\rho\geq 0$ be fixed, Proposition \ref{local_non_expansiveness} implies the existence of a radius $\alpha>0$ such that $\mathcal{A}(x) \subseteq \mathcal{A}_{\rho}(x^*)$ for every $x \in \mathcal{X} \cap \mathbb{B}_{2}(x^*\,;\,\alpha)$. By Definition \ref{def:active_set} of \emph{active set}, $$F(x) = \min_{i\,\in\,[n]}\,f^{(i)}(x) = \min_{i\, \in \,\mathcal{A}(x)}\, f^{(i)}(x)\quad \forall x \in \mathcal{X} \cap \mathcal{D}^{([m])}.$$
Combining both last statements, we can write that for any $x \in \mathcal{X} \,\cap\, \mathbb{B}_{2}(x^*\,;\,\alpha) \,\cap\, \mathcal{D}^{([m])}$,
$$F(x) = \min_{i\, \in \,\mathcal{A}_\rho(x^*)}\, f^{(i)}(x).$$\vspace{-10pt}
Moreover, we notice that 
\begin{align*}\min_{x \in \mathcal{X} \,\cap\, \mathbb{B}_{2}(x^*\,;\,\alpha)}\,F(x)&= \min_{x \in \mathcal{X} \,\cap\, \mathbb{B}_{2}(x^*\,;\,\alpha) \,\cap\, \mathcal{D}^{([m])}}\,F(x)\\
&=\min_{i\, \in \,\mathcal{A}_\rho(x^*)}\, \min_{x \in \mathcal{X} \,\cap\, \mathbb{B}_{2}(x^*\,;\,\alpha) \,\cap\, \mathcal{D}^{([m])}}\,f^{(i)}(x)\\
&=\min_{i\, \in \,\mathcal{A}_\rho(x^*)}\,\min_{x \in \mathcal{X} \,\cap\, \mathbb{B}_{2}(x^*\,;\,\alpha)}\,F^{(\{i\},[m])}(x)\\
&\geq \min_{i\, \in \,\mathcal{A}_\rho(x^*)}\,\min_{x \in \mathcal{X}}\,F^{(\{i\},[m])}(x)=\min_{i\,\in\,\mathcal{A}_{\rho}(x^*)}\,\nu^{(i)}.
\end{align*}
Condition \eqref{eq:scan} finally tells us that 
$$\min_{x \in \mathcal{X} \,\cap\, \mathbb{B}_{2}(x^*\,;\,\alpha)}\,F(x) \geq F^{(\{i^*\},[m])}(x^*)\geq F(x^*).$$
The last inequality comes from the fact that $x^* \in \text{dom}\,F^{(\{i^*\},[m])}$ hence $\text{dom}\,F$. \\Since $x^* \in \mathcal{X} \,\cap\, \mathbb{B}_{2}(x^*\,;\,\alpha)$, we also have that 
$$\min_{x \in \mathcal{X} \,\cap\, \mathbb{B}_{2}(x^*\,;\,\alpha)}\,F(x) \leq F(x^*)$$.
We can conclude that $x^*$ locally minimizes $F$ over $\mathcal{X} \,\cap\, \mathbb{B}_{2}(x^*\,;\,\alpha)$.
\end{proof}
\remark \label{rem:tightness}It can also be shown, independently, thanks to a simple \emph{reductio ad absurdum}, the \emph{piece} $i^*$ that generated $x^*$ actually belongs itself to the set $\mathcal{A}(x^*)$. 

\section{Upper-Lower Optimization algorithm}
\label{ulo_disclosure} 
We postpone the full description of our Upper-Lower Optimization algorithm (\texttt{ULO}), to explain first what are the observations that led us out to devise it as it is.
\subsection{Rationale}  \label{sec:rationale}
    Looking back at the motivating examples and, more broadly, at the core definition of \eqref{eq:min_problem}, one could ask which \emph{functional constraints} really influence the value of $F^*$. \\ \\Recalling Assumption \ref{A1} (B), the minimization of any \emph{piece} $f^{(i)}$ on $\mathcal{X}\,\cap \,\mathcal{D}^{([m])}$ yields a well-defined $x^*(i) := x^*(\{i\},[m])$. The set of \emph{active constraints} at $x^*(i)$ is
    \begin{equation}
        \mathcal{C}(i) = \{j \in [m]\,|\,c^{(j)}(x^*(i)) = 0\}.
        \label{eq:active_constraints}
    \end{equation}
    Removing the \emph{passive constraints} does not alter the optimization,\,i.e.
    \begin{equation}
    \min_{x \in \mathcal{X} \,\cap\,\mathcal{D}^{([m])}}\,f^{(i)}(x) = \nu^{(i)} = \min_{x \in \mathcal{X} \,\cap\,\mathcal{D}^{(\mathcal{C}(i))}}\,f^{(i)}(x).
    \label{eq:neutral}
    \end{equation}
        The naive \texttt{enumeration} visits the candidates $x^*(i)$ for $i=1,\dots,n$ without keeping in memory which \emph{constraints} were active at previously computed $x^*(1), ,\dots,x^*(i-1)$.\\A global minimum $x^*=x^*(i^*)$ of \eqref{eq:min_problem} is found among $\{x^*(i)\}_{i=1}^{n}$ for a certain $i^* \in [n]$.  \\ \\Two favourable situations can occur at this stage. When constraints are fairly different from each other, e.g.\@ affine and linearly independent, most likely only a small subset of them will be active at a given $x^*(i)$. For instance, if the problem defining $x^*(i)$ is an LP, usually $|\mathcal{C}(i)| \in \Theta(d)$ \cite{LY21}, which can be significantly smaller than $m$. \\It could also happen that sets $\mathcal{C}(i)$ share a substantial part of their indices in common, lowering considerably the cardinality of their union \eqref{eq:union_C}. Based on both these statements, if one was able to detect the overall union of \emph{active constraints} sets \begin{equation}
        \bar{\mathcal{C}} := \bigcup_{i=1}^{n}\,\mathcal{C}(i),
        \label{eq:union_C}
        \end{equation} 
        then even the baseline \texttt{enumeration} scheme efficiency could be enhanced in the sense that every $x^*(i)$ could be obtained for less computational efforts as 
        $$x^*(i) = \arg \min_{x \in \mathcal{X} \,\cap\,\mathcal{D}^{(\bar{\mathcal{C}})}}\,f^{(i)}(x) = x^*(\{i\},\bar{\mathcal{C}}).$$
        In this paper, we intend in this paper to iteratively build a set $S \subseteq [m]$ of indices that will serve as surrogate to $\bar{\mathcal{C}}$. Roughly speaking, the minimization of upper models \eqref{eq:upper_model} will give us new candidates $x^*(\bar{i})$ with new active indices $\mathcal{C}(\bar{i})$ to be added to our set $S$. Furthermore, we also keep track of visited indices in $H \subseteq [n]$ to avoid testing a \emph{piece} $i \in [n]$ twice.  The minimization of lower models \eqref{eq:lower_model} includes a minimization pass over every \emph{piece} not in $H$ but only taking into account constraints belonging to set $S$, ideally close from set $\bar{\mathcal{C}}$. By minimizing both upper and lower models, we obtain accordingly upper and lower bounds on $F^*$ that hopefully allow us to skip testing most of the \emph{pieces} in the \texttt{enumeration} phase.

\subsection{Correctness}
\label{sec:correctness}
Now that the rationale behind \texttt{ULO} has been enlightened, we detail our algorithm.\\ We start by presenting its pseudo-code in Algorithm \ref{alg:ulo}. It takes as arguments 
a \emph{piece} index $\hat{i} \in [n]$ to start with as well as two parameters $\tilde{\epsilon} \geq 0$ and $\epsilon\geq0$, respectively relative and absolute approximate globality tolerances. As previously announced, each outer-iteration of \texttt{ULO} involves two successive distinct steps, namely phase (a) (Algorithm \ref{alg:phase_a}) and phase (b) (Algorithm \ref{alg:phase_b}), that we thoroughly explain and analyze. Each time phase (a) is exited, $\hat{F}$ (smallest upper bound computed so far) and the output iterate $x$ are updated whereas at the end of phase (b), $\check{F}$ (biggest lower bound computed so far) is revised and, possibly, approximate global optimality is detected according to the gap $\hat{F}-\check{F}$. After a maximum of $n$ (respectively $m$) updates, $\hat{F}$ (respectively $\check{F}$) must be equal to $F^*$, ultimately pushing the gap $\hat{F}-\check{F}$ towards $0$. 
\vspace{5pt}
\begin{algorithm} [H]
\caption{$\texttt{ULO}(\hat{i},\tilde{\epsilon},\epsilon)$}\label{alg:ulo}
\begin{algorithmic}[1]
\Require $\rho\geq0$, $k=0$,\,$H_0 = \emptyset=S_0$,\,$(\hat{F}_0,\check{F}_0)=(-\infty,\infty),\,\tilde{\epsilon}\geq0,\,\epsilon\geq0$.
\Ensure $x_{k}$ \emph{local minimum} such that $F(x_k)-F^*\leq \max\{\epsilon,\max\{1,|F(x_k)|\} \cdot \tilde{\epsilon}\}$.
\While{$\hat{F}_k-\check{F}_k> \max\{\epsilon,\max\{1,|\hat{F}_k|\} \cdot \tilde{\epsilon}\}$}
\State $(i^*_k,H_{k+1},\hat{F}_{k+1},x_{k+1}) \gets \texttt{phase-(a)}(\hat{i}_k,H_k,\hat{F}_k,x_k)$
\State $(\hat{i}_{k+1},S_{k+1},\check{F}_{k+1}) \gets \texttt{phase-(b)}(i^*_k,S_k,\check{F}_k,H_k)$
\If{$\check{F}_{k+1} = \hat{F}_{k+1}$} 
\Return $x_{k+1}$ \Comment{\textcolor{purple}{\emph{exact global optimality}}}
\Else
\State $k \gets k+1$
\EndIf
\EndWhile\\
\Return $x_k,\hat{F}_k,\check{F}_k$
\end{algorithmic}
\end{algorithm}
\vspace{5pt}
\noindent We dig now into phases (a) \& (b) specifically. For the sake of clarity, $k$ subscripts (the outer-iteration count) were dropped within their respective pseudo-code.  \newpage

\paragraph*{Phase (a)} From an input $V \subseteq [n] \backslash H$, one selects uniformly at random the index $\bar{i}$ of a \emph{piece} not tested yet. According to Assumption \ref{A1} (B), this leads to a well-defined new candidate $\tilde{x}=x^*(\bar{i})=x^*(\{\bar{i}\},[m])$ as the result of the minimization of $F^{(\{\bar{i}\},[m])}$ over $\mathcal{X}$ or, equivalently, the minimization of $f^{(\bar{i})}$ over $\mathcal{X} \cap \mathcal{D}^{([m])}$. The set of tested pieces, $H$, is then updated accordingly and we start over this process until there is \textbf{no new} reachable \emph{piece} from $\bar{x}$ (output iterate of phase (a)), i.e.\@ $\mathcal{A}_{\rho}(\bar{x})\backslash H = \emptyset$. We accept a new candidate if $\nu^{(\bar{i})}$ is strictly smaller than $R$, i.e. the objective value achieved by the last accepted candidate $\bar{x}$. $R$ is initialized at $\infty$ to make the first inner-iteration pass.

\begin{algorithm} [H]
\caption{| $\texttt{phase-(a)} (\hat{i},H,\hat{F},x)$}\label{alg:phase_a}
\begin{algorithmic}[1]
\Require set of visited \emph{pieces} $H$, starting \emph{piece} $\hat{i} \not\in H$, $x$ best \emph{local minimum} encountered so far (or initial iterate), $\hat{F}$ upper bound on $F^*$.
\Ensure $i^*$ such that no $\bar{i} \in \mathcal{A}_{\rho}(\bar{x})$ with $\nu^{(i)}<F(\bar{x})$, the sound update of $(H,\hat{F},x)$. 
\State $(H_+,V,R) \gets (H,\{\hat{i}\},\infty)$\Comment{\textbf{init.}}
\While{$V \not= \emptyset$} 
\State $\bar{i} \sim \text{Uni}\big(V)$
\State $(\tilde{x}, H_+) \gets (x^*(\bar{i}),H_+ \cup \big\{\bar{i}\big\}\big)$
\If{$\nu^{(\bar{i})}\geq R$} 
 $V \gets V\backslash \{\bar{i}\}$
\Else \State $(\bar{x},V,R,i^*) \gets (\tilde{x},\mathcal{A}_{\rho}(\bar{x}) \backslash H_+,F(\tilde{x}),\bar{i})$
\EndIf
\EndWhile 
\If{$F(\bar{x})<\hat{F}$}  $(\hat{F},x) \gets (F(\bar{x}),\bar{x})$ \Comment{update of $\hat{F}$}
\EndIf\\
\Return $(i^*,H_+,\hat{F},x)$
\end{algorithmic}
\end{algorithm}

\noindent When a new step is taken within the \texttt{while} loop of phase (a), a  \emph{piece} $\bar{i} \in \mathcal{A}_\rho(\bar{x})$ is tested. We show in Lemma \ref{lemma_um} that its value $\nu^{(\bar{i})}$ cannot exceed $F(\bar{x})$ by more than $\rho$. 
\begin{lemma}
\label{lemma_um}
Let $\bar{i} \in \mathcal{A}_\rho(\bar{x})$ for any $\bar{x} \in \mathcal{X}\,\cap\,\mathcal{D}^{([m])}$ and $\rho \geq 0$. Then we have 
\begin{equation}
F(x^*(\bar{i})) \leq \nu^{(\bar{i})} = f^{(\bar{i})}(x^*(\bar{i}))\leq f^{(\bar{i})}(\bar{x})\leq F(\bar{x})+\rho. 
\label{eq:cert_decrease}
\end{equation}
\end{lemma}
\begin{proof} 
By definition of $\rho$\emph{-active set} at $\bar{x}$, $f^{(\bar{i})}(\bar{x})$ is at most $\rho$ above the value $\min_{i \in [n]}\,f^{(i)}(\bar{x}) = F(\bar{x})$. Then, since $x^*(\bar{i})$ realizes the minimum of \emph{piece} $\bar{i}$ over $\mathcal{X} \cap \mathcal{D}^{([m])}$, one must have $\nu^{(\bar{i})} = f^{(\bar{i})}(x^*(\bar{i}))\leq f^{(\bar{i})}(\bar{x})$. Finally, one observes that $f^{(\bar{i})}(x^*(\bar{i})) \geq \min_{i \in [n]}\, f^{(i)}(x^*(\bar{i})) = F(x^*(\bar{i}))$ to conclude.
\end{proof}

\noindent %Note that since $\bar{x}$ itself was the minimizer of a tested \emph{piece} (say $\square \in [n]$) within phase (a), Lemma \ref{lemma_um} implies that if the candidate $\tilde{x}=x^*(\bar{i})$ is accepted (i.e. $\nu^{(\bar{i})}<R=F(\bar{x})$) then $\nu^{(\bar{i})}< \nu^{(\square)}$. 
Exiting the \texttt{while} loop of phase (a), one among two possible outcomes occurs:
\begin{itemize}
\item $F(\bar{x}) < \hat{F}$. In this case, we can conclude that $\bar{x}$ is a \emph{local minimum} of \eqref{eq:min_problem}.\\ Indeed, we know that $\bar{x}=x^*(i^*)$ for some $i^*\in H_{+}\backslash H$. Let's assume \emph{ad absurdum} that $\bar{x}$ is not a \emph{local minimizer}. As a consequence of Proposition \ref{local_non_expansiveness}, there exists $\alpha >0$ such that on $\mathcal{X} \,\cap\,\mathbb{B}_2(\bar{x}\,;\,\alpha)\, \cap\, \mathcal{D}^{([m])}$, $F = \min_{i\,\in\,\mathcal{A}_\rho(\bar{x})}\,f^{(i)}$. Since $\bar{x}$ is not \emph{locally optimal}, there must be a \emph{piece} $\tilde{i} \in \mathcal{A}_{\rho}(\bar{x})$ whose minimum over any $\mathcal{X} \,\cap\, \mathbb{B}_2(\bar{x}\,;\, \alpha)\, \cap\, \mathcal{D}^{([m])}$ strictly falls below $F(\bar{x})$. Thereby, it would also mean that $\nu^{(\tilde{i})} < F(\bar{x})$ since $\nu^{(\tilde{i})}$ is the optimal value of $f^{(\tilde{i})}$ over $\mathcal{X}\, \cap\, \mathcal{D}^{([m])}$. Then, either $\tilde{i} \in H$ and, by construction, one would have had $\hat{F}_k \leq F(x^*(\tilde{i}))< F(\bar{x})$ (contradiction) or $\tilde{i} \not\in  H$ but then $x^*(\tilde{i})$ should have been visited ($\tilde{i} \in \mathcal{A}_{\rho}(\bar{x})$) and selected ($\nu^{(\tilde{i})} < F(\bar{x})$) so that $\bar{x}$ could not be the output iterate of phase (a) (contradiction). As pointed out in Remark \ref{rem:tightness}, it also holds $\nu^{(i^*)}=F(\bar{x})$.\\

\item $F(\bar{x}) \geq \hat{F}$. In this case, $\bar{x}$ might possibly be a \emph{local minimum} of \eqref{eq:min_problem} but we already recorded a candidate $x$ associated with a better functional value.
\end{itemize}
To summarize, prior to entering phase (b), \texttt{ULO} ensures that output iterate $x_k$ is the best \emph{local minimum} visited so far and $\hat{F}_k = F(x_k) = \min_{i \in H_k}\,\nu^{(i)}$. Obviously, since $|H_{k+1}\backslash H_k|\geq 1$, there can be at most $n$ passes through phase (a) in total during an execution of Algorithm \ref{alg:ulo}. Hence there must be  exist $\hat{k} \in [n]$ such that $\hat{F}_{\hat{k}} = F^*$.
\vspace{-10pt}
\paragraph*{Phase (b)} Once the set of considered \emph{constraints} $S_+$ is enriched with \emph{active constraints} $\mathcal{C}(i^*)$ at the output iterate $\bar{x}$ of phase (a) as well as with a random $\bar{j}\in [m]\backslash S$ (hence $S\subsetneq S_{+}$), we look after a new \emph{piece} $\hat{i}$ to restart phase (a). To that end, we compute the values $\check{\nu}^{(\hat{i})}$ for every $i \in [n]$, each of which representing a lower bound on $\nu^{(i)}$, as highlighted hereafter: 
\begin{equation}
\nu^{(i)} \geq \check{\nu}^{(i)} = \begin{cases} \nu^{(i)} & i \in H_+ \\ \nu(\{i\},S_{+}) & i \not \in H_+. \end{cases}
\vspace{-6pt}
\end{equation}
Then, we select $\hat{i}\in[n]$ as the most promising \emph{piece}, i.e. $\hat{i} = \arg\min_{i\,\in\,[n]}\,\check{\nu}^{(i)}$, achieving the lower bound $\min_{i \in [n]}\,\check{\nu}^{(i)} \leq F^*$. 
\begin{algorithm} [H]
\caption{| $\texttt{phase-(b)} (i^*,S,\check{F},H_+)$}\label{alg:phase_b}
\begin{algorithmic}[1]
\Require $i^*$ output \emph{piece} from phase (a), $S$ set of previously stored \emph{active constraints}, $\check{F}$ lower bound on $F^*$,$H_+$ set of visited \emph{pieces} (i.e. $\nu^{(i)}$ is known for every $i\in H_+$).
\Ensure $\hat{i} \not\in H_+$ such that 
\State $\bar{j} \sim \text{Uni}([m]\backslash S)$
\State $S_+ \gets S \cup \mathcal{C}(i^*)\cup \{\bar{j}\}$  \Comment{\textbf{init.}}
\State $\hat{i} = \arg\min_{i\in [n]}\,\check{\nu}^{(i)} = \begin{cases} \nu^{(i)} & i \in H_+ \\ \nu(\{i\},S_+) & i \not \in H_+\end{cases}$ 
\State $\check{F} \gets \max\{\check{F},\check{\nu}^{(\hat{i})}\}$\Comment{update of $\check{F}$}\\
\Return $(\hat{i},S_+,\check{F})$
\end{algorithmic}
\end{algorithm}
\noindent At the end of phase (b), one can face the following scenario: $\hat{i}_{k+1} \in H_{k+1}$.\\
Recalling that $\hat{F}_{k+1} = \min_{i\in H_{k+1}}\,\nu^{(i)}$, we can write $$F^* \geq \min_{i \in [n]}\,\check{\nu}^{(i)} =  \check{\nu}^{(\hat{i}_{k+1})} = \nu^{(\hat{i}_{k+1})}\geq \hat{F}_{k+1}\geq F^* \Rightarrow x_{k+1} \in \arg \min_{x\in\mathcal{X}}\,F(x).
$$
Note that for every $\hat{i}_{k+1} \in [n]$, one can safely update $\check{F}_{k+1}= \max\big\{\check{F}_k,\check{\nu}^{(\hat{i}_{k+1})}\big\}. $\\

\noindent In summary, either \texttt{ULO} certifies that $x_{k+1}$ is \emph{globally optimal} or phase (a) is launched again with starting \emph{piece} $\hat{i}$. $\check{F}_{k+1}$ is monotonically increasing. Since $|S_{k+1}\backslash S_k|\geq 1$, there can be at most $m$ phases (b). That is, there exists $\check{k} \in [m]$ such that $S_{\check{k}+1} = [m]$ and $\check{F}_{\check{k}+1} = \min_{i \in [n]}\,\check{\nu}^{(i)} = \min_{i \in [n]}\,\nu(\{i\},S_{\check{k}+1}) = \min_{i \in [n]}\, \nu^{(i)} = F^*$.

\paragraph*{Illustration} 
Let us illustrate the behaviour of \texttt{ULO} on an example, closely related to Example \ref{example:pl}. The original problem is a Tikhonov-regularized piecewise-linear program
\vspace{-17pt}
\begin{equation}
\min_{|u|\leq 5} \,\frac{u^2}{2} +\bigg[ \max_{j \in [6]} \, \beta_1^{(j)} \cdot u + \gamma_1^{(j)} \bigg] -\bigg[ \max_{i \in [6]} \, \beta_2^{(i)}\cdot u + \gamma_2^{(i)} \bigg]
\label{eq:ulo_1D_raw}
\vspace{-5pt}
\end{equation}
with $(\beta_1^{(1)},\dots,\beta_1^{(6)})= \big(\frac{1}{4},-\frac{1}{2},\frac{1}{3},2,0,3\big)$, $(\gamma_1^{(1)},\dots,\gamma_1^{(6)})= \big(-2,-1,0,-2,-\frac{1}{4},-4\big)$ and  $(\beta_2^{(1)},\dots,\beta_2^{(6)})= \big(\frac{3}{2},1,-1,4,-2,0\big)$, $(\gamma_2^{(1)},\dots,\gamma_2^{(6)})= \big(0,2,1,-1,1,2\big)$.\\

\vspace{-2pt}
\noindent Translated into the \eqref{eq:min_problem} framework, the problem now reads
\vspace{-10pt}
\begin{equation*}
\min_{x=(u,\eta) \,\in \,\mathbb{B}_\infty(0\,;\,5)\times \mathbb{R}} \quad   \min_{i\in[6]}\bigg\{\frac{u^2}{2} + \eta - \beta_2^{(i)}\cdot u - \gamma_2^{(i)}\bigg\}\quad
\textrm{s.t.} \quad  \max_{j\in[6]}\bigg\{\beta_1^{(j)}\cdot u + \gamma_1^{(i)} - \eta \bigg\} \leq 0%\label{eq:ulo_1D}  
\end{equation*}
We display on Figure \ref{fig:visu_ulo2} two iterations of \texttt{ULO} applied on the above problem. The values of different models are computed on $\mathcal{R}=\{(u,\eta)\,|\,\eta = \max_{j \in [6]}\, \beta_1^{(j)}\cdot u + \gamma_1^{(j)},\,|u|\leq 8\}$.
\vspace{-15pt}
\begin{figure}[ht]%
    \centering
    \vspace{-20pt}
    \subfloat[\centering $H_1 = \{1,4\}$, $\bar{x}_{1}=(2,2)$, $\hat{F}_1 = -3$]{{\includegraphics[width=7cm]{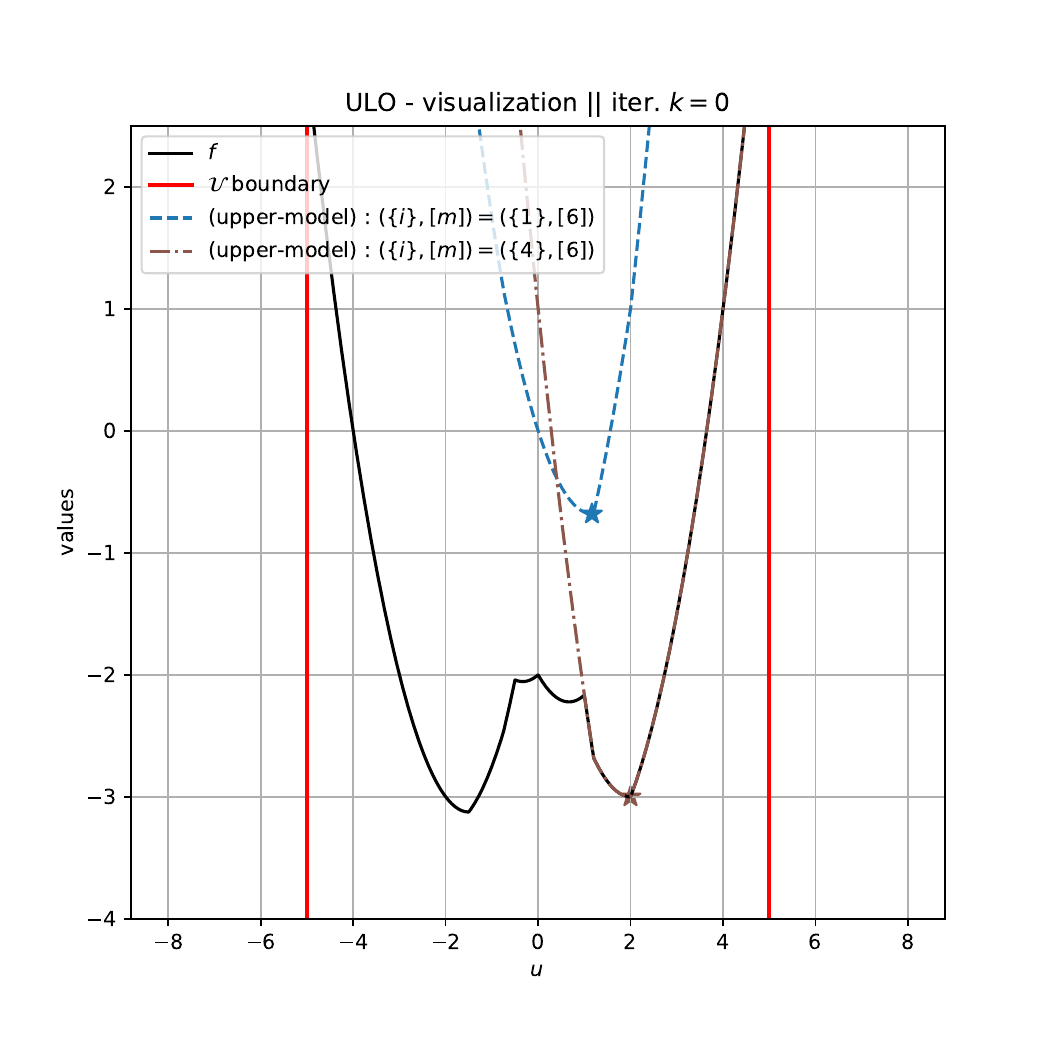} }}%
    \subfloat[\centering $S_1 = \{4, 6\}$, $\check{x}_{1}=(-4,-10)$, $\check{F}_1 = -11$]{{\includegraphics[width=7cm]{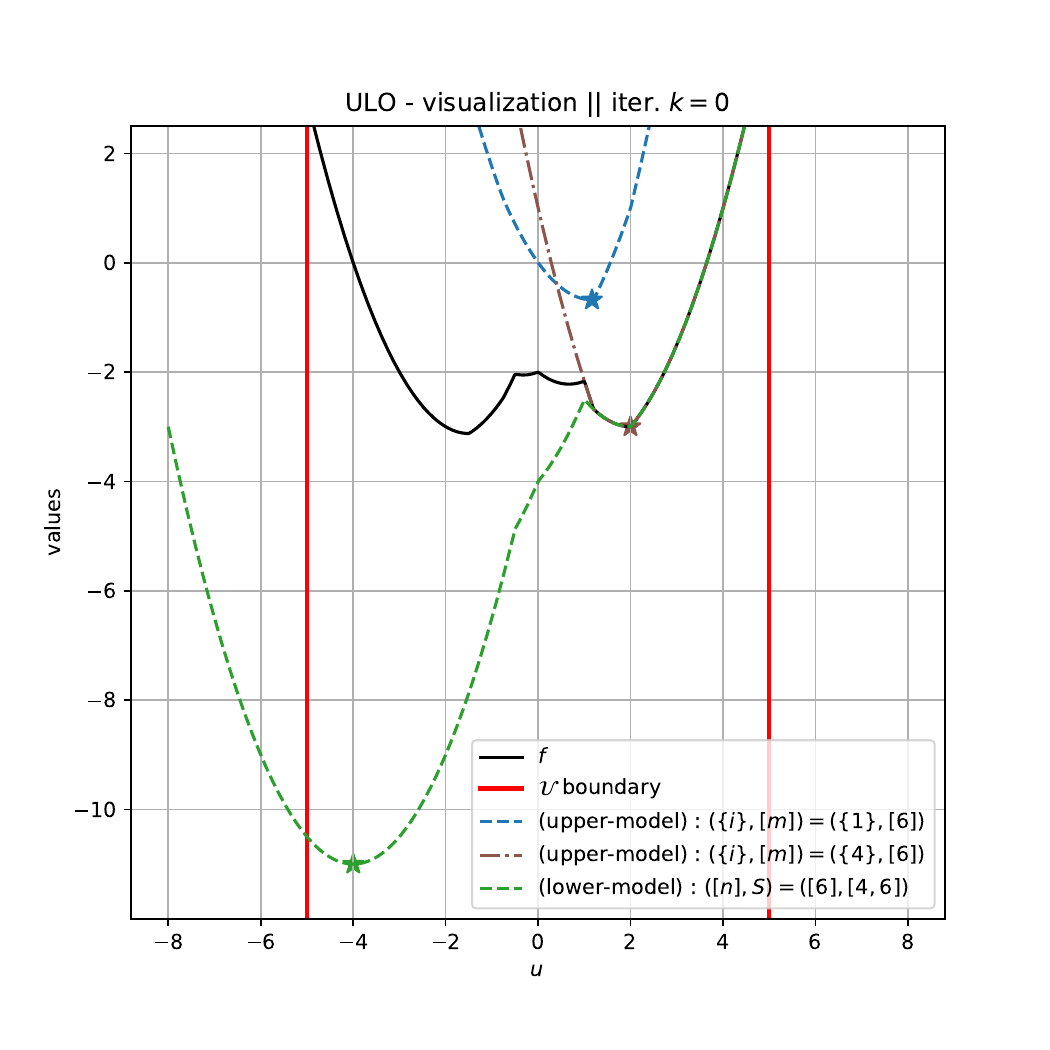} }}% 
    %\caption{TEST}%
    \label{fig:visu_ulo}%
\end{figure}
\vspace{-20pt}
\begin{figure}[ht]%
    \centering
    \vspace{-20pt}
    \subfloat[\centering $H_2 = \{1,4,5\}$, $\bar{x}_{2}=(-3/2,-1/4)$, $\hat{F}_2 = -25/8$]{{\includegraphics[width=7cm]{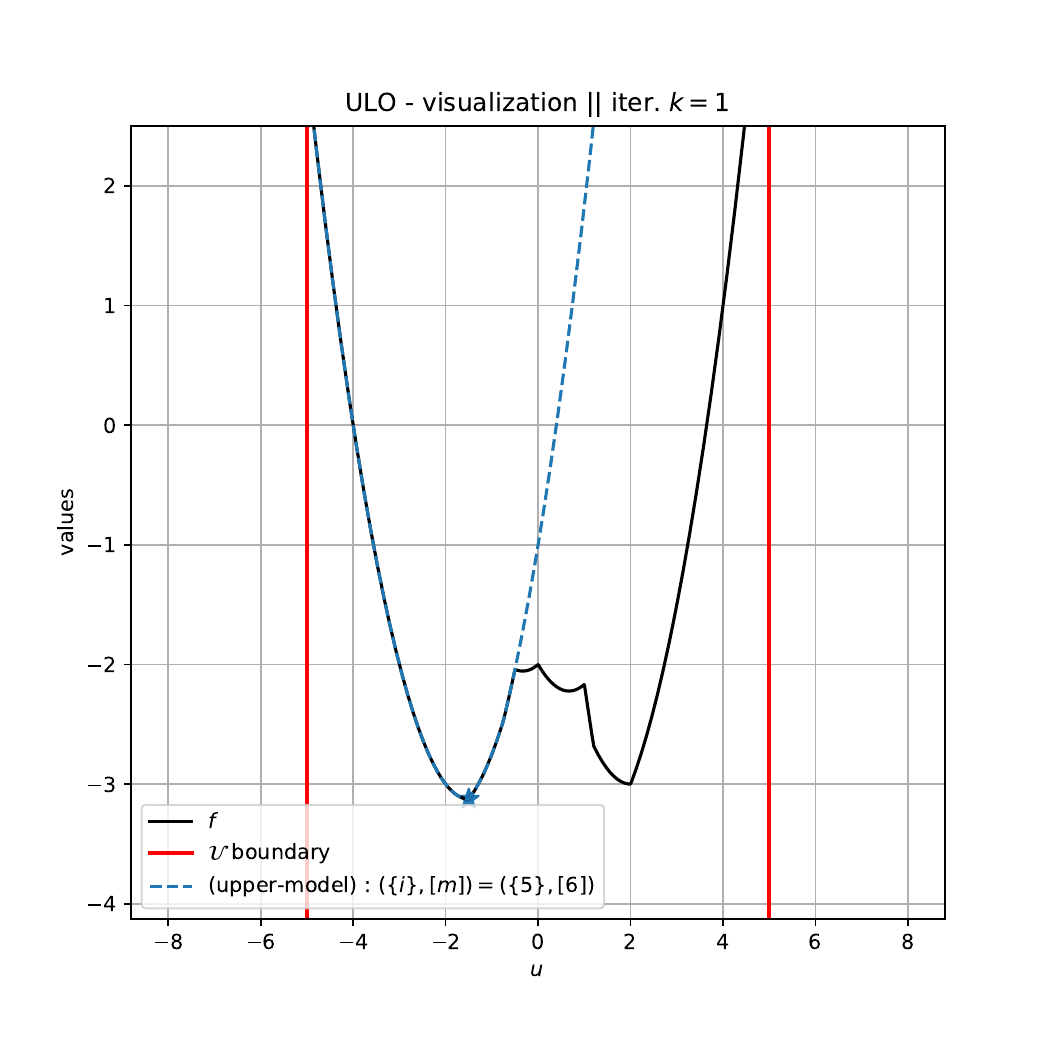} }}%
    \subfloat[\centering $S_2 = \{2,4,5,6\}$, $\check{x}_{2}=(-3/2,-1/4)$, $\check{F}_2 = -25/8$]{{\includegraphics[width=7cm]{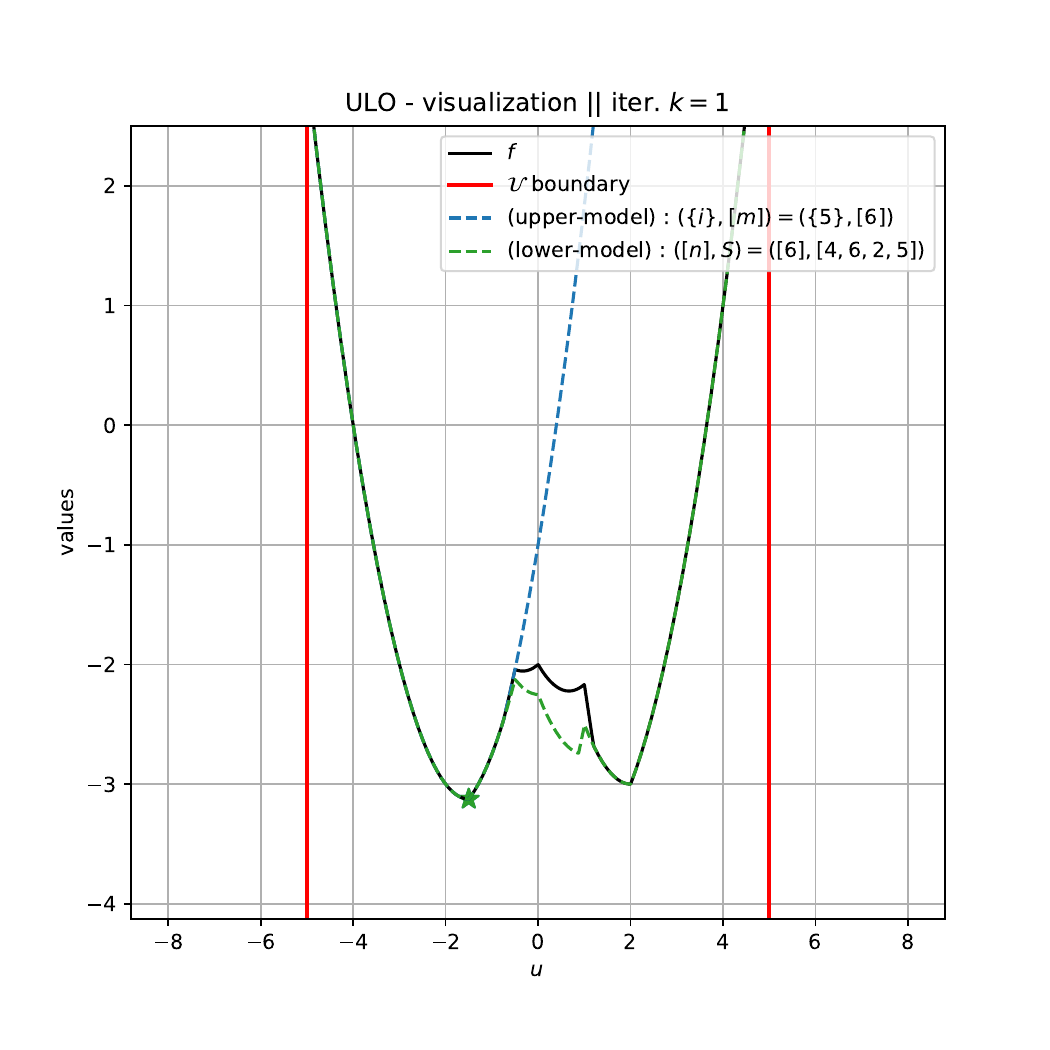} }}% 
    \caption{Representation of $f(u)=F(u,\eta)$ and various iterates/models for every $(u,\eta) \in \mathcal{R}.$\\The gap $\hat{F}_2-\check{F}_2 = 0$ certifies that $x_2 = \bar{x}_2$ is \emph{globally optimal} whereas $x_1 = \bar{x}_1$ was only \emph{locally optimal}.}%
    \label{fig:visu_ulo2}%
\end{figure}
\newpage 
\noindent We now state the main theorem characterizing the convergence of \texttt{ULO}.
\begin{theorem}
\label{theorem_ulo} Let $\tilde{\epsilon}\geq0,\hspace{2pt}\epsilon \geq0$ be prescribed accuracies and let $\hat{i} \in [n]$ be any starting piece to test. $\texttt{ULO}\,(\hat{i},\tilde{\epsilon},\epsilon)$ returns a point $x_K \in \mathcal{X}$ such that \begin{equation}
    F(x_K)-F^* \leq \max\big\{\epsilon, \max\{1,|F(x_K)|\} \cdot \tilde{\epsilon}\big\}
\label{eq:opt_gap}
\end{equation}
in $K\leq \min\{n,m+1\}$ iterations, each of which involving a finite number of oracle calls. Moreover, $x_K$ is a local optimum of \eqref{eq:min_problem}.
\end{theorem}
\begin{proof} We have already proven that $x_k$ is maintained as the best \emph{local minimum} encountered across the iterations and therefore, as suggested by Remark \ref{rem:tightness}, $\hat{F}_k = F(x_k)$. By construction, $\check{F}_k \leq F^*$ for any $k \in \mathbb{N}$ and thereby, 
$$F(x_k)-F^* = \hat{F}_k-F^* \leq \hat{F}_k-\check{F}_k.\vspace{-5pt}$$
We have shown that sequence $\{\hat{F}_k\}_{k \in \mathbb{N}}$ (respectively $\{\check{F}_k\}_{k \in \mathbb{N}}$) was monotonically decreasing (respectively increasing) so that the optimality gap is monotonically decreasing with $k \in \mathbb{N}$. Moreover, we noticed that in any case, $\hat{F}_k = F^*$ for every $k \geq n$ and $\check{F}_k = F^*$ for every $k \geq m$. Considering the first case, if $k=n$, it is certain that $\check{F}_n = F^* = \min_{i \in [n]}\,\nu^{(i)}$ since $H_{n} = [n]$. On the other hand, as previously mentioned, it is sure that $S_{k} = [m]$ for $k \geq m$. In such circumstance, $\check{F}_k = F^*$ since for any $i \not \in H_k$, $\check{\nu}^{(i)} = \nu(\{i\},S_k) = \nu(\{i\},[m]) = \nu^{(i)}$ (and, by definition, $\check{\nu}^{(i)} = \nu^{(i)}$ for every $i \in H_k$) thus $\check{F}_k = \min_{i \in [n]}\,\check{\nu}^{(i)} = \min_{i \in [n]}\,\nu^{(i)} = F^*$. Yet, is still possible that $F(x_k)=\hat{F}_k$ prevents \texttt{ULO} to exit at this iteration. Nevertheless, the suggested piece $\hat{i} = \arg \min_{i \in [n]}\,\check{\nu}^{(i)} \not \in H_k$ will lead to the next iterate $x_{k+1}$ being \emph{globally optimal}. Hence the closure of the optimality gap is guaranteed after $\min\{n,m+1\}$ iterations. The estimated gap $\hat{F}_k-\check{F}_k$ can fall below any level $\delta = \max\big\{\epsilon, \min\{1,|F(x_k)| \cdot \tilde{\epsilon}\}\big\}$ of approximate global optimality in $K < \min\{n,m+1\}$ iterations.
\end{proof}
\vspace{-20pt}
\subsection{DAG interpretation} 
\label{subsec:dag}
Let $\rho \geq0$ be fixed. Throughout this section, we assume for simplicity\begin{equation}i_+ \in\mathcal{A}_{\rho}(x^*(i))\,\wedge \,i_+\not=i\, \Rightarrow\, \nu^{(i_+)}<F(x^*(i))\leq \nu^{(i)}\quad \forall (i, i_+) \in [n]^2. \label{eq:dag_assumption}\end{equation} We remind that Lemma \ref{lemma_um} implies $\nu^{(i_+)} \leq f^{(i_+)}(x^*(i)) \leq F(x^*(i))+\rho\leq \nu^{(i)}+\rho$. \\That is, we consider that a \emph{piece} $i_+$ (nearly) achieving the minimum value among the $n$ \emph{pieces} at $x^*(i)$ (i.e. $F(x^*(i))$), when optimized itself on the feasible set $\mathcal{X}\,\cap\,\mathcal{D}^{([m])}$, yields $\nu^{(i_+)}$ strictly improving with respect to $F(x^*(i))$. Bearing this in mind, one can build a simple \emph{directed acyclic graph} $G([n],E)$ that summarizes what are the possible transitions from one candidate $i$ to another $i_+$ within \texttt{ULO}. Its vertices are the \emph{pieces} $i\in[n]$ with values $\nu^{(i)}$ and its set of edges $E$ is defined as 
\begin{equation}
    E := \big\{(i,i_+) \in [n]^2 \,|\,i_+ \in \mathcal{A}_{\rho}(x^*(i))\,\wedge\,i_+\not= i \big\}.
    \label{eq:graph_E}
\end{equation}
One can observe that, indeed,  \eqref{eq:dag_assumption} and \eqref{eq:graph_E} prevent $G$ from containing cycles. If a node $i_+$ is accessible from $i$ in $G$ then $F(x^*(i_+))\leq \nu^{(i_+)} < F(x^*(i))$. Then, it is impossible to trace back a path from $i_+$ to $i$ following vertices $\tilde{i}$ with values $F(x^*(\tilde{i})) < F(x^*(i_+))$. Let $W\subseteq [n]$. Let  $G[W]$ depict the subgraph of $G$ induced by vertices in $W$.\\ For every node $i$ of graph $G[W]$, we denote
$\textbf{outdeg}(i) = \big|\{i_+ \in W\,|\,(i,i_+)\,\in\,E\}\big|$.\newpage 
\noindent
Starting from a node $\hat{i} \in [n]$, phase (a) performs a random walk in $G[[n]\backslash H]$, until a \textcolor{blue}{sink node $i^*$} is reached (i.e. $i^* \in [n]$ with $\textcolor{blue}{\textbf{outdeg}(i^*) = 0}$), see Figure \ref{fig:dag} \& \ref{fig:dag2}. \\

\noindent At the end of phase (a), $H$ is augmented and now includes all the nodes visited alongside the directed path $(\hat{i} \to i^*)$. After phase (b), a new random walk is undertaken with the elected $\hat{i} \in [n]\backslash H$ and the overall procedure is run again.
\example In order to substantiate the present section, we expose the situation below with $n = 9$ and $\rho$ fixed in $[0,1/10]$. As usual, $\nu^{(i)} = f^{(i)}(x^*(i))$ for every $i \in [n]$.
\begin{table}[ht] 
\centering
\renewcommand*{\arraystretch}{1.35}
\begin{tabular}{|c||c|c|c|}
\hline
   \textbf{nodes} $i$ & $\rho$-\emph{active sets} $\mathcal{A}_\rho(x^*(i))$ & values of \emph{active pieces} at $x^*=x^*(i)$ & $\nu^{(i)}$\\
    \hline
    $1$ & $\{1\}$ & $f^{(1)}(x^*) = 5/2$ & $5/2$ \\
    $2$ & $\{4\}$ & $f^{(4)}(x^*) = 5/2$& $16/5$ \\
    $3$ & $\{3\}$ &  $f^{(3)}(x^*)=2$ & $\textcolor{ForestGreen}{2}$\\
    $4$ & $\{3,5\}$ &$f^{(3)}(x^*) = 8/3$, $f^{(5)}(x^*) = 8/3$  & $ 3$\\
    $5$ & $\{5\}$ & $f^{(5)}(x^*)= 7/3$&$7/3$\\
    $6$ &$\{4,9\}$ & $f^{(4)}(x^*) = 9/2$, $f^{(9)}(x^*) = 9/2-\rho/3$& $6$\\
    $7$ & $\{1,2,3\}$ & $f^{(1)}(x^*)= 4$, $f^{(2)}(x^*)= 4$, $f^{(3)}(x^*)= 4$ & $6$\\
    $8$ & $\{8\}$ & $f^{(8)}(x^*)= 5/2$ & $5/2$\\
    $9$ & $\{9\}$ & $f^{(9)}(x^*) = 4$& $4$ \\
    \hline
\end{tabular}
\vspace{5pt}
\caption{Summary of key characteristics of a possible \eqref{eq:min_problem} problem | \textcolor{ForestGreen}{$F^* = \nu^{(3)} = 2$}}
\label{tab:dag_example}
\end{table}

\noindent From this table, we can draw diagrams of Figure \ref{fig:dag} \& \ref{fig:dag2}. \\

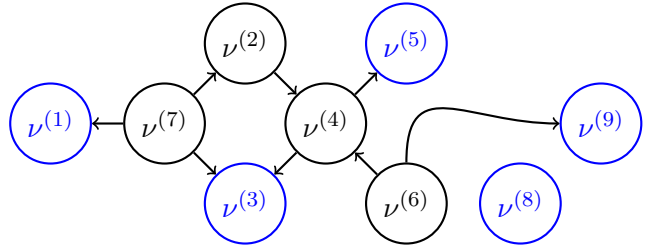
\begin{figure}[h!]
\vspace{-10pt}
   \centering
  \hspace{50pt}
\begin{tikzpicture}[node distance={15mm}, thick, main/.style = {draw, circle}] 
\node[main] (1) {$\nu^{(7)}$}; 
\node[main,color=blue!] (0) [left of=1] {$\nu^{(1)}$};
\node[main] (2) [above right of=1] {$\nu^{(2)}$}; 
\node[main,color=blue!] (3) [below right of=1] {$\nu^{(3)}$}; 
\node[main] (4) [above right of=3] {$\nu^{(4)}$}; 
\node[main,color=blue!] (5) [above right of=4] {$\nu^{(5)}$}; 
\node[main] (6) [below right of=4] {$\nu^{(6)}$}; 
\node[main,color=blue!] (7) [right of=6] {$\nu^{(8)}$};
\node[main,color=blue!] (8) [above right of=7] {$\nu^{(9)}$};
\draw[->] (1) -- (2); 
\draw[->] (1) -- (3); 
\draw[->] (6) to [out=90,in=180,looseness=1.5] (8); 
%\draw[->] (1) to [out=180,in=270,looseness=5] (1); 
\draw[->] (4) -- (3);
\draw[->] (2) -- (4); 
\draw[->] (4) -- (5); 
\draw[->] (1) -- (0);
%\draw[->] (5) to [out=315, in=315, looseness=2.5] (3); 
\draw[->] (6) --  (4); 
\end{tikzpicture} 
\caption{initial graph | $G[[n]\backslash \emptyset]=G([n],E)$.}
\label{fig:dag}
\end{figure}
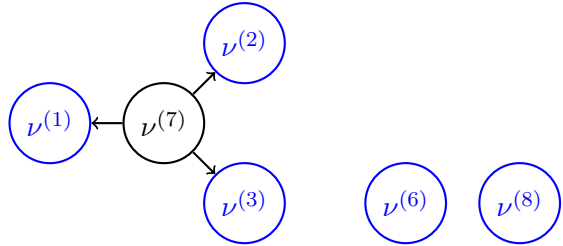
\begin{figure}[h!]
\vspace{-10pt}
   \centering
  \hspace{50pt}
\begin{tikzpicture}[node distance={15mm}, thick, main/.style = {draw, circle}] 
\node[main] (1) {$\nu^{(7)}$}; 
\node[main,color=blue!] (0) [left of=1] {$\nu^{(1)}$};
\node[main,color=blue!] (2) [above right of=1] {$\nu^{(2)}$}; 
\node[main,color=blue!] (3) [below right of=1] {$\nu^{(3)}$}; 
\node[main,color=blue!] (6) [below right of=4] {$\nu^{(6)}$}; 
\node[main,color=blue!] (7) [right of=6] {$\nu^{(8)}$};

\draw[->] (1) -- (2); 
\draw[->] (1) -- (3); 
%\draw[->] (1) to [out=180,in=270,looseness=5] (1); 
\draw[->] (1) -- (0);
%\draw[->] (5) to [out=315, in=315, looseness=2.5] (3); 
\end{tikzpicture} 
\caption{subgraph after two walks: $
\{4, 5\}$ and $\{9\}$ | $G[[n]\backslash H]$ with $H = \{4,5,9\}$.}
\label{fig:dag2}
\end{figure}
\remark No edge $(1,8)$ is present since $\nu^{(8)} = f^{(8)}(x^*(8)) = 5/2 \not< F(x^*(1)) = 5/2$. By \eqref{eq:dag_assumption}, $f^{(8)}(x^*(1))>5/2+\rho$ because otherwise, $8 \in \mathcal{A}_{\rho}(x^*(1))$ but then $\nu^{(8)}<5/2$.

\newpage 

\subsection{Efficiency simulations} 
As previously teased in our introduction (Section \ref{intro_notes}), the purpose of this section is to analyze, \emph{a priori}, what are the parameters, characterizing the structure and inherent difficulty of an instance of \eqref{eq:min_problem}, for which \texttt{ULO} is expected to outperform the baseline, i.e. the pure \texttt{enumeration strategy} as in \eqref{eq:enumeration_scheme_ref}.\\

\noindent First, we need to state another assumption regarding our oracle time complexity. 

\begin{assumption}
\label{A2} Let $i \in [n]$ and $S \subseteq [m]$.  
\begin{itemize}
\item[] There exists a strictly increasing function $T^{(i)} : [m] \to \mathbb{R}_+$ that outputs the wall-clock time $T^{(i)}(|S|)$ taken by our black-box oracle \eqref{eq:small_min_problem_i} to return $x^*(\{i\},S)$.
\end{itemize}
\end{assumption}

\remark Through Assumption \ref{A2}, we implicitly consider that the computational burden to compute $x^*(\{i\},S)$ is only influenced by the selected \emph{piece} $i \in [n]$ and the number of \emph{constraints} kept in $S \subseteq [m]$. The value of $T^{(i)}$ definitely depends on the solvers at hand, e.g.\@ active-set methods \cite{Wong11}, simplex methods \cite{Spielman04}, interior-point methods \cite{Nemi08}, etc. Yet, the nature of \emph{pieces}, e.g.\@ whether $\{(x,\nu) \in \mathcal{X}\times \mathbb{R}\,|\,f^{(i)}(x)\leq \nu\}$ is a polyhedral cone, second-order cone, exponential cone, etc., plays the most significant role in the definition of the function $T^{(i)}$ once the solver has been chosen.\\

\noindent At this stage, the total time taken by \texttt{ULO} (stopping after $K \in \mathbb{N}$ iterations) and the \texttt{enumeration strategy}, solely based on Assumption \ref{A2}, are as follows:
\begin{table}[ht]
\centering
\renewcommand*{\arraystretch}{1.85}
\begin{tabular}{|l|c||c|}
\hline
\textbf{Algorithm} & \texttt{Enumeration} & \texttt{ULO}\\
    \hline
    \hline
\textbf{Running time} & $\underbrace{\sum_{i\, \in \,[n]}\,T^{(i)}(m)}_{\text{full enumeration}}$ & $\underbrace{\sum_{i \in H_K}\,T^{(i)}(m)}_{\text{phase (a)}} + \underbrace{\sum_{k=1}^K\,\sum_{i\,\in\, [n]\backslash H_k}\,T^{(i)}(|S_k|)}_{\text{phase (b)}}$\\
\hline
\textbf{Guarantees} & \emph{global} optimality & $\epsilon$ \emph{global} or $\tilde{\epsilon}$ \emph{relative} optimality gap \eqref{eq:opt_gap} \\
\hline
\end{tabular}
\vspace{10pt}
\caption{General time complexities of \texttt{ULO} and the baseline \texttt{enumeration strategy}.}
\label{tab:complexities}
\end{table}

\noindent For every $k \in \{0,\dots,K-1\}$, let $\Delta H_{k+1} := H_{k+1}\backslash H_k$ contain the new \emph{pieces} tested at iteration $k$ and let $i^*_k$ (respectively $\bar{j}_k$) be the output \emph{piece} of phase (a) (respectively the drawn \emph{constraint} index $\bar{j}$, see line $1$ of Algorithm \ref{alg:phase_b}) during that same iteration. 
As such, without further assumptions on functions $\{T^{(i)}\}_{i\in[n]}$, Table \ref{tab:complexities} is not yet informative enough to predict situations, symbolized by sequences like 
\begin{equation}
\bigg(\Delta H_1,S_1,\dots,\Delta H_{K},S_K\bigg), \quad S_{k+1} = \bigcup_{\tilde{k}=0}^{k}\,\mathcal{C}(i^*_{\tilde{k}})\,\cup\,\{\bar{j}_{\tilde{k}}\} \quad \forall k \in \{0,\dots,K-1\},
\label{eq:gen_scenario}
\end{equation}
for which the expected time complexity of \texttt{ULO} would be smaller than \texttt{enumeration}'s. \\ 

\noindent Therefore, we present Assumption \ref{A3}, narrowing down the scope of possible interpretations for $\{T^{(i)}\}_{i\in[n]}$. As suggested in our introduction and supported by \cite{Arkadi94}, we introduce a simple complexity model mirroring interior-point methods efficiency \cite{Nemi08} since the use of these latter prevails in our numerical experiments.  
\newpage 
\begin{assumption}
\label{A3}  There exists constants $r>0$ and $C >0$ such that 
\begin{equation}
T^{(i)}(|S|) = T(|S|) = C + |S|^{r}\quad\forall i \in [n]. \label{eq:cost_model}
\end{equation} 
\end{assumption}
\vspace{-5pt}
\remark Constant $C$ serves to better capture the practical behaviour of numerical solvers, usually involving a problem-dependent fixed computational cost to store data in memory and/or encode high-level problem's specification as workable data structures (e.g.\@ matrices of a standard conic formulation, see \textit{matrix stuffing} \cite{Wytock16}).
\vspace{-8pt}
\paragraph*{Practical considerations} Although not mentioned textually, for the sake of clarity, within the pseudo-code of Algorithm \ref{alg:ulo}, it is in practice worth allowing \texttt{ULO} to exit earlier as soon as the update of $\hat{F}$ yields a certificate that   $\hat{F}-\check{F} \leq \max\big\{\epsilon, \max\{1,|\hat{F}|\} \cdot \tilde{\epsilon}\big\} $
 holds for the user prescribed accuracies $\tilde{\epsilon}, \epsilon \geq 0$. That is, \texttt{ULO} can be stopped since the output iterate satisfies the (approximate) optimality requirements and no more computational work should be undertaken. Another possible practical modification resides in storing, if the memory space allows it, subsequent oracle outputs $x^*(\{i\},S_{k+1})$ computed at iteration $k$ for every $i \in [n]\backslash H_{k+1}$. These vectors will serve to \emph{warm-start} problems $\nu(\{i\},S_{k+2})$ at the next iteration, i.e. $k+1$, for every $i \in [n]\backslash H_{k+2}$, reducing in practice the time cost of phase (b), theoretically $(n-|H_{k+2}|)\cdot(C+ |S_{k+2}|^r)$.\\ \vspace{-8pt}
\paragraph*{Simulator}
It order to simulate \texttt{ULO} on various (possible) instances of \eqref{eq:min_problem} and then draw lessons about its efficiency, one must build models that describe relationships between all its constitutive elements. For instance, one must quantify how crude lower bounds $\nu(\{i\},S)$ on \emph{pieces}' values $\nu^{(i)}$ are. Another important feature to fix is deciding which and how many constraints are \emph{active} at points $\{x^*(i)\}_{i\in[n]}$. 
Through the lens of the DAG interpretation and according to \eqref{eq:dag_assumption} and \eqref{eq:graph_E}, it amounts to elect outgoing edges in the graph abstraction $G([n],E)$. We specify our model in Table \ref{tab:simulator}. \\

\noindent As usual, $i \in [n]$, $S\subseteq [m]$ and $(\bar{i},\bar{i}_+) \in E$.\\
\vspace{-10pt}
\begin{table}[ht] 
\centering
\renewcommand*{\arraystretch}{1.5}
\begin{tabular}{|c||c|}
\hline
\textbf{expressions} & parameters \\
    \hline    $|\mathcal{C}(i)| =  \varsigma_{\text{act}} \cdot m, \quad  \quad \mathbb{P}[\mathcal{C}(i)] \propto \sum_{j \in \mathcal{C}(i)}\,j^\upsilon$ & $\upsilon>0$, $\varsigma_{\text{act}} \in (0,1)$\\
    $\nu^{(\bar{i}_+)} \leq \nu^{(\bar{i})} - \text{Exp}(\theta)$ & $\theta > 0$  \\
    $\nu^{(i)}-\nu(\{i\},S)= \min\bigg\{ \bigg(1-\frac{|\mathcal{C}(i) \,\cap \,S|}{|\mathcal{C}(i)|}\bigg)\, ,\, \bigg(1-\frac{|S|}{m}\bigg)^{3/2}\bigg\} \cdot \text{Exp}(\bar{\theta})$ & $\bar{\theta}>0$ \\
    $\sum_{i=1}^n\,\textbf{outdeg}(i) = \varsigma_{\text{deg}} \cdot \frac{n\cdot(n-1)}{2}$ & $\varsigma_{\text{deg}} \in (0,1)$
\\
    \hline
\end{tabular}
%\vspace{-5pt}
\caption{Summary of simulator's parametrized stochastic models.}
\label{tab:simulator}
\end{table}
\vspace{-5pt}
\remark \noindent Viewing $\varsigma_{\text{act}}\cdot m$ as the largest size among the \emph{active constraints} sets $\{\mathcal{C}(i)\}_{i\in[n]}$, our model is intentionally adversarial so that $\bar{\mathcal{C}}$ (see \eqref{eq:union_C}) is the biggest possible, hence rendering phase (b) of \texttt{ULO} more expensive. We want our simulated results to be robust against such bad situations so that if \texttt{ULO} wins (i.e. takes less computational time units to exit),  it would win even more in more realistic situations.

\newpage 

Let us explain the role of each parameter involved in Table \ref{tab:simulator}.\vspace{5pt}
\begin{itemize}
    \item Parameter $\upsilon$ induces a selection bias towards certain constraints. The biggest $\upsilon$ is, the less diversity there will be in the $\varsigma_{\text{act}}\cdot m$ constraints of sets $\big\{\mathcal{C}(i)\big\}_{i\in[n]}$.\\
    Obviously the size of $\bar{\mathcal{C}}$ will likely drop with $\varsigma_{\text{act}}$ getting smaller and $\upsilon$ getting bigger. This intuition will be confirmed on Figure \ref{fig:appetizer} (a).
    \\
    \item 
   Parameter $\theta$ serves as a reference quality difference between two adjacent \emph{pieces} in the sense of the DAG interpretation, it can also be seen as the expected functional value gain one will observe by moving from \emph{piece} $\bar{i}$ to $\bar{i}_+$.\\ 
    \item Parameter $\bar{\theta}$ describes how far, on average, is the worst lower bound $\nu(\{i\},\emptyset)$ from $\nu^{(i)}$. In addition, as underlined in \eqref{eq:neutral}, if $S \supseteq \mathcal{C}(i)$, $\nu(\{i\},S)$ must equal $\nu^{(i)}$. Also, the more constraints were added to $S$, the tightest should the lower bounds be. Based on these observations, we derived the third model of Table \ref{tab:simulator}.
    We can expect an abstract instance of \eqref{eq:min_problem} to get more difficult as $\bar{\theta}$ grows.\\
    \item Parameter $\varsigma_{\text{deg}}$ represents the fraction of (directed) edges present in $G([n],E)$ with respect to the complete (undirected) graph with $n\cdot(n-1)/2$ edges. Parameter $\varsigma_{\text{deg}}$ plays a crucial role in the time spent within phase (a) of \texttt{ULO} since the number of edges influences the number of paths within $G([n],E)$ hence the number of consecutive \emph{pieces} one visits at each iteration $k$ of Algorithm \ref{alg:ulo}.\\
\end{itemize}
%\vspace{-10pt}
To create a single (abstract) instance of \eqref{eq:min_problem} in accordance with the specifications of Table \ref{tab:simulator}, we proceed as follows. Implementation details are freely available on \href{https://github.com/guiguiom/ULO_public}{GitHub}.
\begin{enumerate}
\item \textbf{DAG generation} | $G([n],E)$\\
\begin{itemize}
\item We start from scratch, i.e. $E = \emptyset$, and iteratively add new edges to $E$. To that purpose, we sample one edge candidate uniformly at random among remaining possibilities, initially corresponding to the whole set $[n]^2$. That edge is accepted if no cycle is created or removed from the remaining possibilities otherwise. We loop until $\varsigma_{\text{deg}} \cdot n(n-1)/2$ (directed) edges are added.\\
\end{itemize}
\item \textbf{Values assignment} | $\{\nu^{(i)}\}_{i \in [n]}$ \\
\begin{itemize}
\item For every $i \in [n]$ and $(\bar{i},\bar{i}_+) \in E$, we draw independently $\theta_{(\bar{i},\bar{i}_+)} \sim \text{Exp}(\theta)$, $\bar{\theta_{i}} \sim \text{Exp}(\bar{\theta})$ and $\mathcal{C}(i)$ of size $\varsigma_{\text{act}} \cdot m$. This latter is achieved without replacement and following the probability distribution of the first row of Table \ref{tab:simulator}. We point out that this distribution is non-uniform, unless $\upsilon=0$. Then, we solve the LP \eqref{eq:assign_abstract} to obtain $\mathbf{F}^{(i)} = F(x^*(i))$ and $\nu^{(i)}$ for every $i \in [n]$. Without loss of generality, we fixed a reference upper bound $\Lambda = 10$. Finally, we opt for a threshold $\xi \simeq 10^{-7}>0 $ for strict inequalities \eqref{eq:graph_E}.
\begin{equation}
\begin{aligned}
\max_{\mathbf{F} \,\leq\, \nu\, \leq \,\,\Lambda\cdot\mathbf{1}_n} \quad & \langle \mathbf{1}_n, \nu\rangle \quad & 
\label{eq:assign_abstract}\\
\textrm{s.t.} \quad & \nu^{(\bar{i}_+)} \leq \nu^{(\bar{i})}-\theta_{(\bar{i},\bar{i}_+)} \quad & \forall (\bar{i},\bar{i}_+)\in E \\ 
 \quad & \nu^{(\bar{i}_+)} \leq \mathbf{F}^{(\bar{i})}-\xi \quad & \forall (\bar{i},\bar{i}_+)\in E \\ 
  \quad & \nu^{(i)} = \mathbf{F}^{(i)} \quad & \forall i \in [n],\,\textbf{outdeg(}i\textbf{)} = 0. \\ 
\end{aligned}
\end{equation}
\end{itemize}
\end{enumerate}

\newpage 
% tab with nominal values
\begin{figure}[ht]%
    \hspace{-50pt}
    \vspace{-20pt}
    \subfloat[\centering ]{{\includegraphics[scale=0.4]{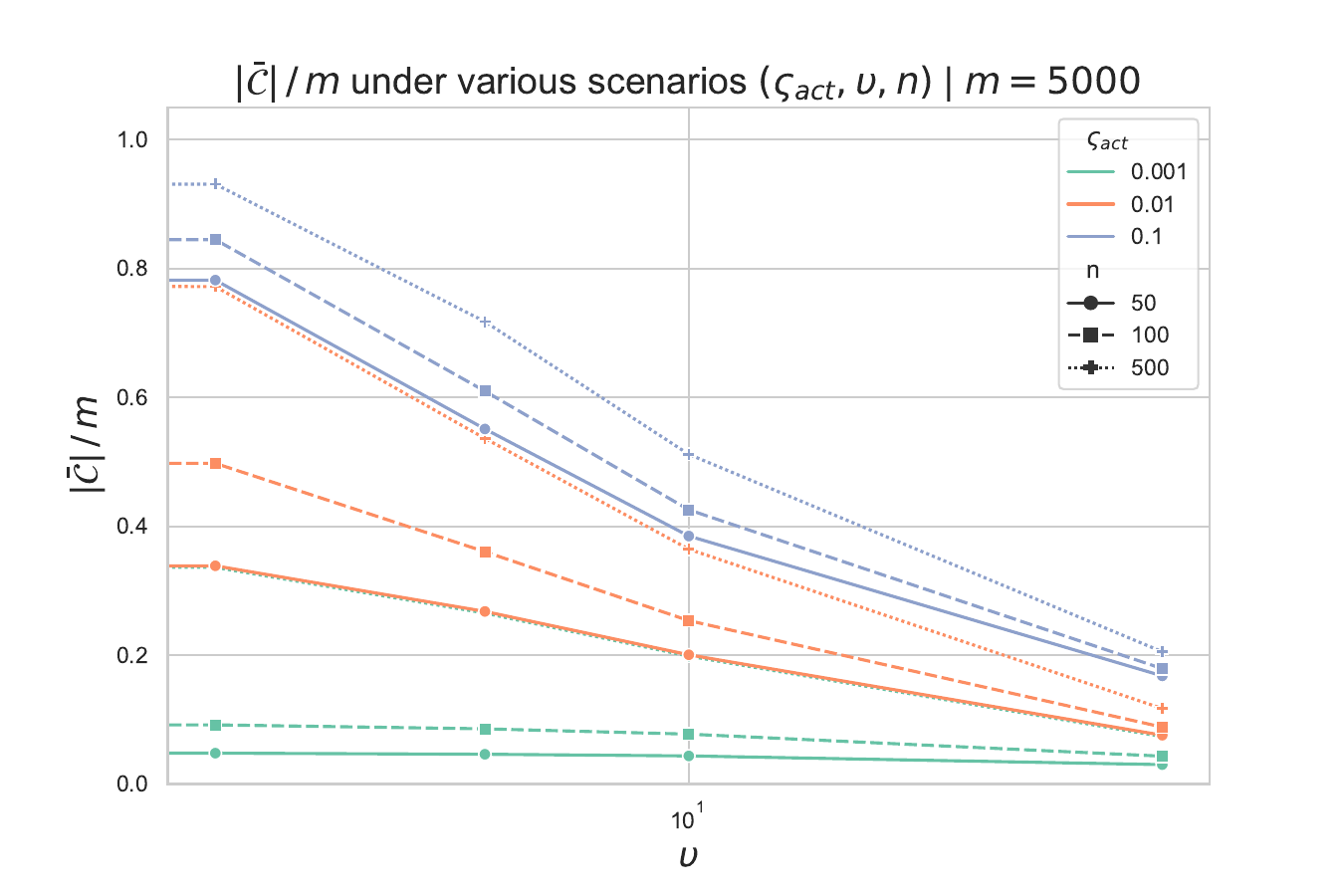} }}%
    \subfloat[]{{\includegraphics[scale=0.4]{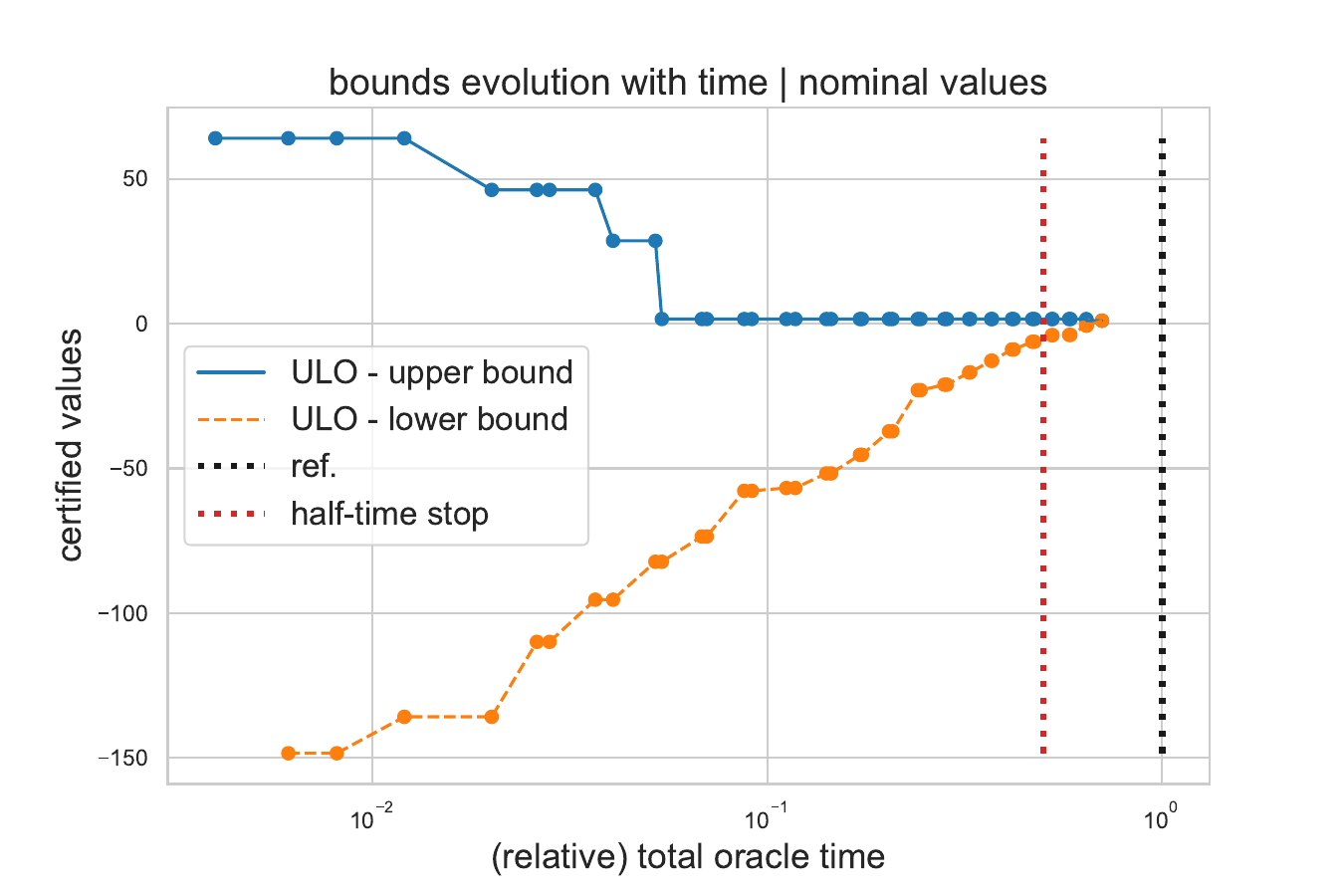} }}% 
    \vspace{10pt}
    \caption{Typical trends.}%
    \label{fig:appetizer}%
\end{figure}
\vspace{-15pt}
\begin{figure}[ht]%
     \hspace{-50pt}
    \vspace{-20pt}
    \subfloat[]{{\includegraphics[scale=0.4]{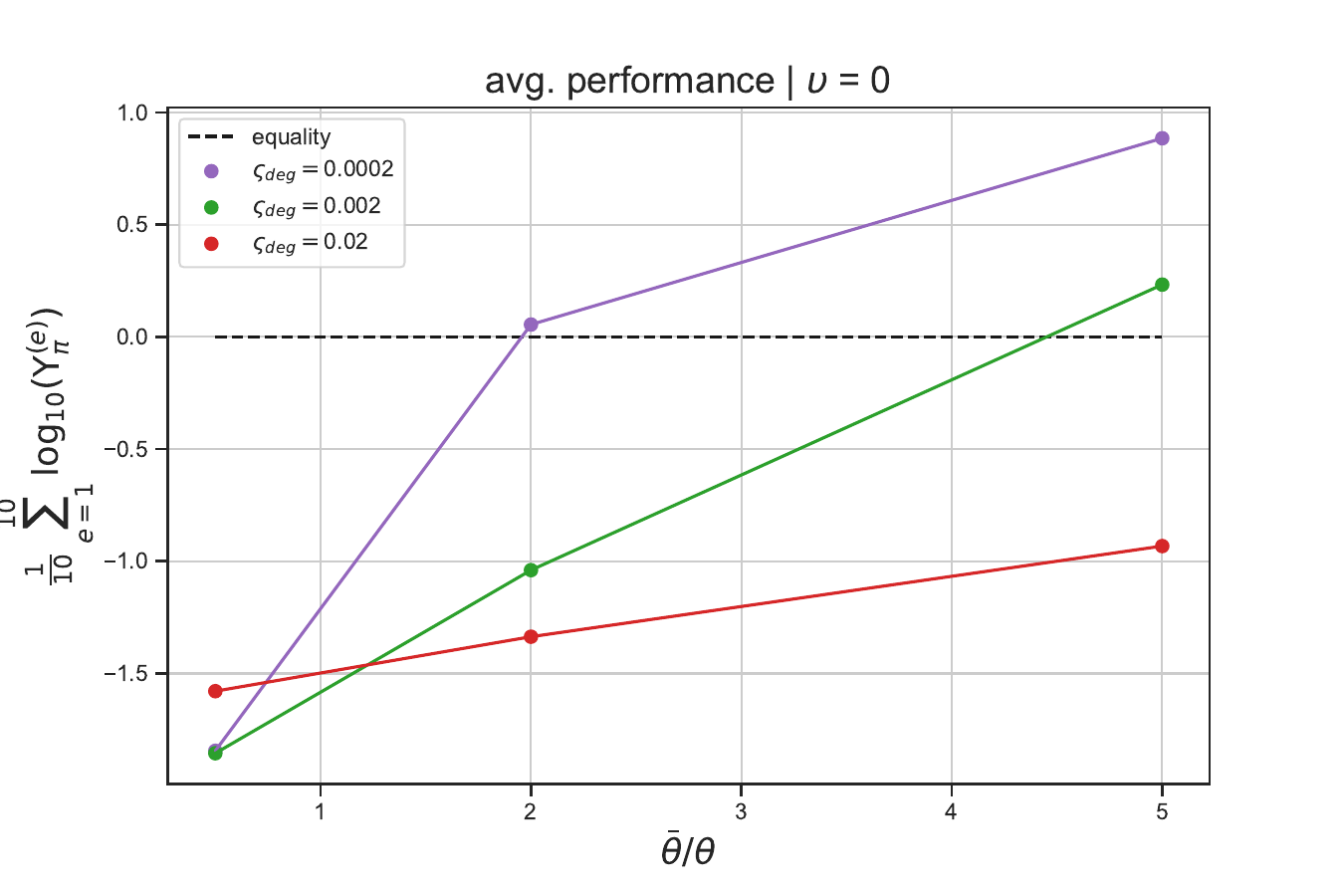} }}%
    \subfloat[]{{\includegraphics[scale=0.4]{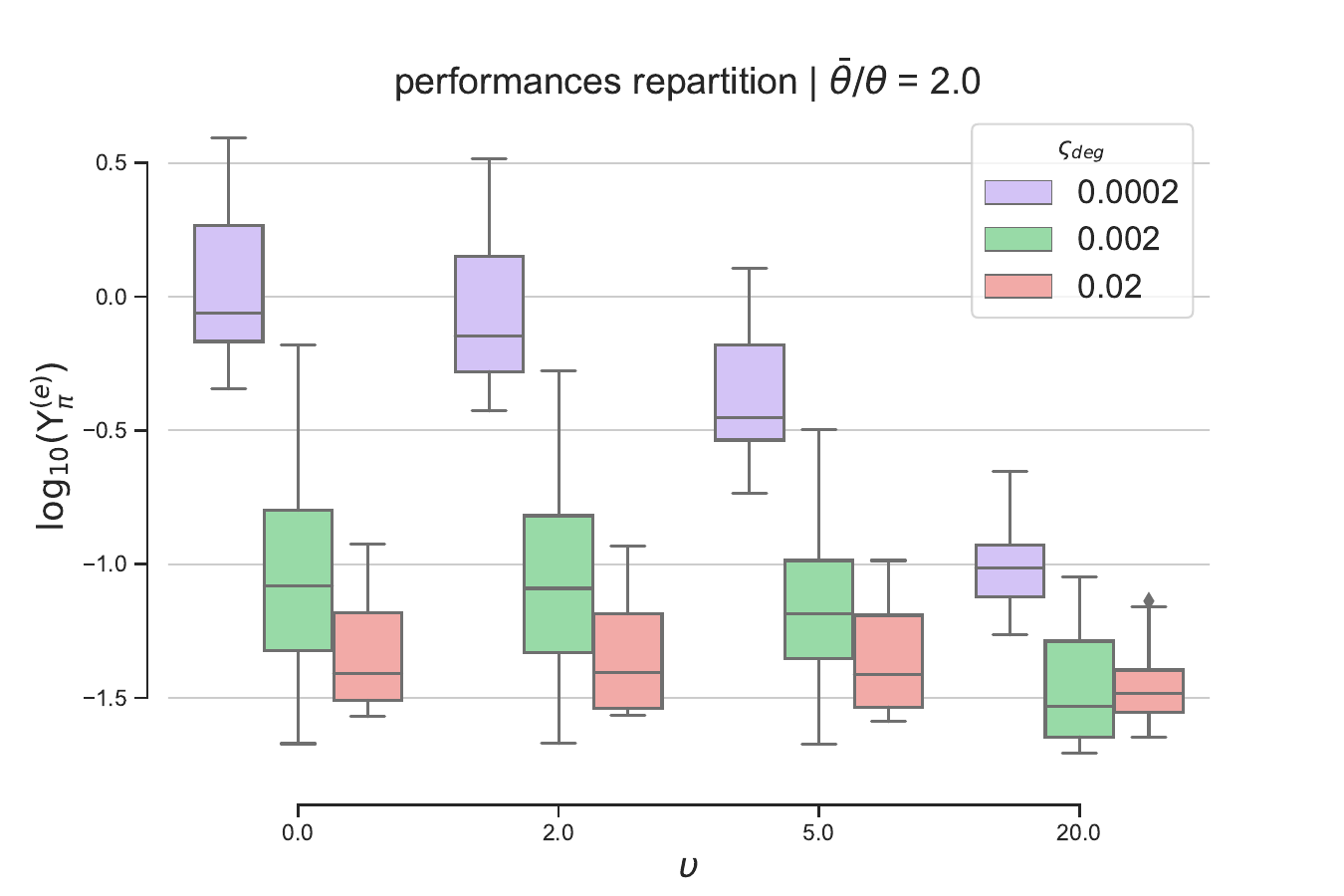} }}% 
    \vspace{10pt}
    \caption{Macroscopic insights.}%
    \label{fig:menu}%
\end{figure}
\noindent Once an abstraction is generated, we can simulate Algorithm \ref{alg:ulo} without having to wait for the computations of $x^*(i)$ or $x^*(\{i\},S)$ for numerous values of $i \in [n]$ and $S \subseteq [m]$ as all we actually care about are the quantities $\nu^{(i)}$, $\nu(\{i\},S)$, $F(x^*(i))$ and sets $\mathcal{C}(i)$. Each time a call to the black-box \eqref{eq:small_min_problem_i}
would have been performed in practice to compute $x^*(\{i\},
S)$, we increment a counter by $T(|S|)= C + |S|^r$ units, per Assumption \ref{A3}. The total time taken by the \texttt{enumeration scheme} and by \texttt{ULO} after $k \in \mathbb{N}$ iterations are \begin{align}
   \Gamma_{\texttt{ES}}(k) &= k \cdot \big(C +m^r\big), \label{eq:cost_es_k}\\
   \Gamma_{\texttt{ULO}}(k) &= |H_k|\cdot \big(C+m^r\big)+\sum_{\tilde{k}=1}^{k}\,(n-|H_{\tilde{k}}|)\cdot\big(C + |S_{\tilde{k}}|^r\big).\label{eq:cost_ulo_k}
\end{align}
where one recalls the identities of Table \ref{tab:complexities}. Let $K \in \mathbb{N}$ be the number of iterations performed in a run of \texttt{ULO}, the ratio between total times reads 
\begin{equation}
\Upsilon = \Gamma_{\texttt{ULO}}(K) / \Gamma_{\texttt{ES}}(n).
\label{eq:ratio_perf}
\end{equation}
\newpage 
\begin{table}[ht] 
\centering
\renewcommand*{\arraystretch}{1.5}
\begin{tabular}{c||c}
\hline
\textbf{dimensions / parameters / tolerances} & nominal values \\
    \hline
$(n,m)$ & $(5 \cdot 10^2,10^4)$\\
\hline
$\varsigma_{\text{act}}$ & $10^{-2}$ \\ 
\hline
$(C,r)$ & $(1118,3/2)$\\
\hline
$\theta$ & $10$\\
\hline
$(\tilde{\epsilon},\epsilon)$ & $(5\cdot 10^{-2},10^{-3})$\\
    \hline
\end{tabular}
\vspace{5pt}
\caption{Nominal values of the parameters of our abstract model.}
\label{tab:simulator_parameters}
\end{table}
\vspace{-20pt}
\paragraph*{Results}
To obtain Figure \ref{fig:appetizer} (b) and Figure \ref{fig:menu} (a) \& (b), we used the dimensions and parameters prescribed in Table \ref{tab:simulator_parameters}. Ranges for other parameters are given below. Although there were several possible values for $(C,r)$, we sticked to $r=3/2$ as it was motivated in our introductory section and led to clear interpretable results. Once $r$ was fixed, we decided that $C$ should match the cost of a thousandth of a full oracle call, i.e. $T(m)$ (see Assumption \ref{A3}). Note that the bigger value of $C$, the best for \texttt{ULO} since it requires fewer outer-iterations than the $n$ outer-iterations of the \texttt{enumeration}.\\

\noindent A single run of \texttt{ULO} is displayed on Figure \ref{fig:appetizer} (b) where one can observe a typical trend. The upper bound $\hat{F}$ reaches $F^* = 1$ quite rapidly, i.e. after $\sim 10^{-1}\,\Gamma_{\texttt{ES}}(n)$ units of time. This behaviour will also be observed in practice in our real numerical experiments (Section \ref{num_experiments}). Here, the triplet $(\bar{\theta},\upsilon,\varsigma_{\text{deg}})$ was fixed to $(3,5,2\cdot10^{-3})$. \\ \\ %instance TO CHANGE 
For graphs in Figure \ref{fig:menu},  we generated $10$ abstract instances of \eqref{eq:min_problem} for each triplet $(\bar{\theta},\upsilon,\varsigma_{\text{deg}}) \in \{5,20,50\} \times \{0,2,5,20\} \times \{2\cdot10^{-4},2\cdot10^{-3},2\cdot10^{-2}\}$ and we averaged the overall (theoretical) times of \texttt{ULO} over $250$ random starting pieces $\hat{i} \in [n]$. That is, for each triplet $\pi = (\bar{\theta},\upsilon,\varsigma_{\text{deg}})$ and each instance $e\in [10]$, we sampled uniformly at random a $250$-subset of $[n]$, namely $P_\pi^{(e)}$, and ran the simulation for each $\hat{i} \in P_\pi^{(e)}$ to obtain $\Gamma_{\texttt{ULO}}(K^{(e)}_\pi(\hat{i}))$ where $K^{(e)}_{\pi}(\hat{i})$ is the number of iterations \texttt{ULO} took to complete when started from $\hat{i}$ under scenario $\pi$ for the $e$-th abstract problem. We finally aggregate and normalize the performances associated with each abstract problem in the following way,
$$\Upsilon^{(e)}_\pi = \Gamma_{\texttt{ULO},\pi}^{(e)}/ \Gamma_{\texttt{ES}}(n) = (\Gamma_{\texttt{ES}}(n))^{-1}\,\sum_{\hat{i}\in P^{(e)}_\pi}\,\big|P_\pi^{(e)}\big|^{-1}\,\Gamma_{\texttt{ULO}}(K^{(e)}_\pi(\hat{i})) \vspace{-10pt}$$
%On Figure \ref{fig:menu}, each data point corresponds then to a single value $\Upsilon^{(e)}_\pi$ among the $10$.\\
The exit time of \texttt{ULO} crucially depends on the quality of lower bounds (tuned by $\bar{\theta}$) with respect to the expected difference $\theta$ between two adjacent \emph{pieces} (cfr. Table \ref{tab:simulator}). Hence, we fixed $\theta$ arbitrarily and analyzed different regimes for $\bar{\theta}$, i.e. $\bar{\theta}/\theta \in \{1/2,2,5\}$. Not surprisingly, for lower values of $\bar{\theta}$, \texttt{ULO} clearly outperforms the \texttt{enumeration scheme}, e.g.\@ $10^{-1}\,\sum_{e=1}^{10}\,\log_{10}(\Upsilon^{(e)}_\pi)<-1$ for $\bar{\theta}/\theta=1/2$ as depicted on Figure \ref{fig:menu} (a) with $\upsilon=0$. This is because one obtains a decent $\check{F}\leq F^*$ even when the set of considered \emph{constraints} $S$ is small, at little computational expense in phase (b) of \texttt{ULO}. How impressive it might already look, we note that this setting (i.e. $\upsilon=0$) translates an uniform sampling among the \emph{constraints} to construct the sets of \emph{active constraints} $\mathcal{C}(i)$ for every $i\in [n]$ so that $|\bar{\mathcal{C}}|$ (size of their union) is the biggest (cfr. Figure \ref{fig:appetizer} (a)), usually detrimental for \texttt{ULO}. Therefore, the actual benefits of \texttt{ULO} grow with the value of $\upsilon>0$ as one can observe on Figure \ref{fig:menu} (b). Nevertheless, on that same Figure \ref{fig:menu} (b), we expose results that are more contrasted as we consider a larger ratio $\bar{\theta}/\theta=2$. When $\varsigma_{\text{deg}}$ is small, $|E|$ drops and the graph $G([n],E)$ is not well connected. Consequently, there exists a fair number of \emph{local-minima} (scaling with nodes $i^*\in[n]$ such that $\textbf{outdeg}(i^*)=0$ in $G([n],E)$) and phase (a) of \texttt{ULO} will encounter difficulties in decreasing $\hat{F}$ multiple times in a row since long paths within $G([n],E)$ are limited by the small quantity of edges. Thereby, in general, \texttt{ULO} takes advantage of a good connectivity between \emph{pieces} of a \eqref{eq:min_problem} problem, represented by a graph $G([n],E)$ with more edges. Interestingly enough when $\bar{\theta}/\theta=1/2$, on Figure \ref{fig:menu} (a), the relative performance of \texttt{ULO} with respect to \texttt{enumeration} is predicted to be worst for the largest value of $\varsigma_{\text{deg}}$. We explain it as follows. Looking closely at the recorded data, for almost every simulation with $\pi=(5,0,\varsigma_{\text{deg}})$, there were only $K=2$ outer-iterations meaning that problems were really easy for \texttt{ULO}. Phase (b) likely provided a good $\hat{i}\in[n]$ at the end of the first outer-iteration based on which the second call to phase (a) led from $\hat{i}$ to $i^*$ with $i^*$ being an optimal \emph{piece}. However, both calls to phase (a) were presumably more expensive when $\varsigma_{\text{deg}}$ was the largest since $G([n],E)$ is more connected.
\paragraph*{Takeaway message} A clear correlation between $|\bar{\mathcal{C}}|$ and the efficiency of \texttt{ULO} can be established and is driven by parameters $(\varsigma_{\text{act}}, \upsilon)$: the smaller $|\bar{\mathcal{C}}|$, the better \texttt{ULO} is expected to perform. Moreover, \texttt{ULO}'s efficiency heavily depends on the quality of lower bounds $\{\nu(\{i\},S)\}_{i\,\in\,[n]}$ on \emph{pieces}' values $\{\nu^{(i)}\}_{i\,\in\,[n]}$ for sets $S \subseteq [m]$. When $\nu^{(i)}-\nu(\{i\},S)$ vanishes relatively quickly to $0$ for $|S|\ll m$ then \texttt{ULO} should beat the \texttt{enumeration} strategy.
\vspace{-10pt}
\section{Numerical Experiments}
\label{num_experiments}
In this last section, we benchmark three algorithms to solve actual (generalized) problem instances based on Example \ref{example:orlp}. Aside from \texttt{ULO} and the \texttt{Enumeration Scheme} (\texttt{ES}), we also put into test a third procedure, \texttt{Restart Alternating Minimization} (\texttt{RAM}), which is a mix between \texttt{ULO} and \texttt{ES}.
\paragraph*{Restarted Alternating Minimization}
As developed in the previous section, starting from a piece $i \in [n]$, one can jump to another piece $i_+ \in \mathcal{A}_{\rho}(x^*(i))$ while expecting a functional decrease or, at least, a functional increase of a most $\rho$ (see Lemma \ref{lemma_um}). \\We have observed that iterating these jumps until we are stuck amounts to a random walk in a particular directed (sub)graph made of all the nodes matched with \emph{pieces} that were not visited yet. This corresponds to phase (a) of \texttt{ULO}. Actually, when $\rho \to 0$, it can be seen as an alternating minimization scheme. To see this, we reformulate the discrete minimum over the $n$ \emph{pieces} using variables $q \in \Delta^n$ and rewrite \eqref{eq:min_problem} as 
\begin{equation}
F^* = \min_{x \,\in\, \mathcal{X}\, \cap\, \mathcal{D}^{([m])},\,q\, \in \,\Delta^n}\,\sum_{i=1}^{n}\,q^{(i)}\,f^{(i)}(x)
\end{equation}
Every \emph{piece} $i \in [n]$ is uniquely encoded by $q^{(i)} = 1$ and $q^{(i')}=0$ for every $i' \in [n]\backslash \{i\}$. Let $x \in \text{dom }F$ be fixed. If $i_+ \in \mathcal{A}_0(x)$, the vector $q_+$ that encodes $i_+$ (as described in the last sentence) must satisfy 
$$q_+ \in \arg \min_{q\, \in \,\Delta^n}\,\sum_{i=1}^{n}\,q^{(i)}\,f^{(i)}(x).$$
Conversely, fixing $q \gets q_+$ and minimizing over $x$ this time, $x_{++} = x^*(i_+)$ satisfies
$$x_{++} \in \arg \min_{x \,\in \,\mathcal{X}\, \cap \,\mathcal{D}^{([m])}}\,\sum_{i=1}^{n}\,q_+^{(i)}\,f^{(i)}(x).$$
Phase (b) has been devised to produce meaningful bounds on $F^*$ and provide a clever restart of phase (a), which are the main principles behind \texttt{ULO}. One could adopt a strategy that, instead of executing phase (b) as in \texttt{ULO}, randomly restart phase (a) when reaching a sink. Unlike the \texttt{Enumeration Scheme} that tries every piece $i \in [n]$ in a totally random order (i.e. computes the values $\{\nu^{(i)}\}_{i \in [n]}$ one-by-one according to a given permutation of $[n]$ to solve \eqref{eq:min_problem}), one can expect to initialize phase (a) with a piece $\hat{i} \in [n]$ and follow a directed path leading, in just a few alternating minimization steps, to an optimal \emph{piece} $i^*$. Therefore, we name this intermediate strategy \texttt{Restarted Alternating Minimization} (\texttt{RAM}). We suspect that it will reach, on average, an upper bound $\hat{F} = F^*$ quicker than (\texttt{ES}). To point out the eventual benefits of phase (b) alone, we also account for $\rho=10^{-3}$ in phase (a) of \texttt{Restarted Alternating Minimization}, just like for \texttt{ULO}. They only differ by the way they restart their local-search phase. That being stated, the effect of a small $\rho$ remains, however, limited and we do not observe substantial variations when letting $\rho \to 0$.

%\paragraph*{Disclaimer}
%For the sake of transparency, we remind that all the codes producing the results we are going to display are available on \href{https://github.com/guiguiom/ULO_public}{GitHub}.

\subsection*{Experiment | Pessimistic-Optimistic Piecewise-Linear Programming} % note on reuse of lower-bounds for different values of omega :) 

We chose a parametric piecewise-linear optimization problem. Our problem originates from the continuous interpolation between \emph{optimistic} \cite{Norton2017} and \emph{pessimistic} perturbed linear programming arising in market investment optimization. That is, one disposes of $n$ market scenarios $\{(\beta^{(i)},\gamma^{(i)})\}_{i\in[n]}$ stemming from possible deviations with respect to a \emph{nominal market} $(\bar{\beta},\bar{\gamma}) \in \mathbb{R}^p_+ \times \mathbb{R}_{+}$. From each tuple $(\beta^{(i)},\gamma^{(i)})$, $\beta^{(i)} = (\beta^{(i)}_{1},\dots,\beta^{(i)}_p)$ describes the expected returns for assets $1$ to $p$ of market $i \in [n]$, expressed for one unit of investment whereas $\gamma^{(i)}$ denotes a fixed fee cost entering that specific market. In a \emph{pessimistic} setting, one would like to maximize the expected profit assuming that the worst market scenario would occur. On the contrary, the \emph{optimistic} setting considers that one can choose which market conditions will apply and then optimize the investments accordingly. Let $R>0$ be the amount of investment units at disposal. We also take into account $m$ maximal investment conditions $\{(v^{(j)},W^{(j)})\}_{j \in [m]}$, i.e. $\langle v^{(j)},u\rangle \leq W^{(j)}$ for every $j \in [m]$. Finally, it is allowed for these constraints to be violated by a positive amount of at most $\eta$. This latter will be (heavily) penalized in the objective by $-C_{\text{pen}} \cdot \eta$ ($C_{\text{pen}} \sim 5\cdot 10^4$ in our experiments). The full problem reads
\begin{equation}
\begin{aligned}
\max_{(u,\eta) \,\in \,\mathcal{U}\times\mathbb{R}_+} \quad &  -C_{\text{pen}}\cdot \eta +\omega\cdot\bigg(\min_{\tilde{i}\in[n]}\, \langle \beta^{(\tilde{i})},u\rangle -\gamma^{(\tilde{i})}\bigg) + (1-\omega)\cdot \bigg(\max_{i\in[n]}\, \langle \beta^{(i)},u\rangle-\gamma^{(i)}\bigg)
\label{eq:exp_pl_opt}
\\
\textrm{s.t.} \quad & \max_{j\in [m]}\,\langle v^{(j)},u\rangle -W^{(j)}-\eta\leq 0 
\end{aligned}
\tag{POPLP}
\end{equation}
\\
with parameter $\omega \in [0,1]$ allowing a trade-off between \emph{pessimistic} ($\omega=1$) and \emph{optimistic} ($\omega=0$) settings. For instance, one could optimize $u$ under the hypothesis that there is an equal chance that the \emph{pessimistic} or the \emph{optimistic} scenario occurs, i.e. $\omega = 1/2$. Note that if $\omega=1$, \eqref{eq:exp_pl_opt} becomes convex. Thus, we focus on $\omega \in [0,1)$.\begin{remark} We can reformulate \eqref{eq:exp_pl_opt} so that it fits the \eqref{eq:min_problem} framework with
\begin{table}[ht] 
\centering
\renewcommand*{\arraystretch}{1.5}
\begin{tabular}{|c||c|}
\hline
   \textbf{elements} & expression  \\
    \hline
    $d$ & $p+1$  \\
    \hline
    $\mathcal{X}$ & $\mathcal{U} \times \mathbb{R}_+$ \\
    \hline
    $f^{(i)}$ & $C_{\text{pen}}\cdot\eta +\omega \cdot\big(\max_{\tilde{i}\in[n]}\,\gamma^{(\tilde{i})}- \langle \beta^{(\tilde{i})},u\rangle\big)+(1-\omega)\cdot\big(\gamma^{(i)}-\langle \beta^{(i)},u\rangle\big)$ \\
    $c^{(j)}$ & $ \langle v^{(j)} ,u \rangle - W^{(j)}-\eta $ \\
    \hline
\end{tabular}
\vspace{7pt}
\caption{\eqref{eq:exp_pl_opt} translated as an instance of \eqref{eq:min_problem}.}
\label{tab:polp_prog}
\end{table}
\end{remark}
\vspace{-25pt}
\subsection*{Data Generation}
\noindent We give now the explicit procedure that produced our (synthetic) data for \eqref{eq:exp_pl_opt}.\\

\begin{itemize}
\item 
For every $i \in [n]$, the data has been generated identically and independently as follows. The \emph{nominal scenario} is elected as follows. We set $\bar{\gamma}=0$ and for every entry $l \in [p]$, $\bar{\beta}_l \sim \text{Uni}([15])$. Then, $\Delta^{(i)}_l$ is drawn with respect to
$$\mathbb{P}[\Delta_l^{(i)}=\delta] = \begin{cases}3/5 & \delta=0  \\ 1/10 & \delta\in\{\pm 1 ,\pm 2\} \\ 0 & \text{otherwise}. \end{cases} $$
Moreover, let $I\in\{1,10\}$ be a scaling parameter, we set $\beta^{(i)} = \max\{\mathbf{0}_p,\bar{\beta} + I\cdot\Delta^{(i)}\}$ and $\gamma^{(i)} \sim \text{Uni}([0,\gamma_{\text{max}}])$ with 
$$\gamma_{\text{max}} := \frac{15}{100}\cdot\Big(\max_{\tilde{i}\in[n]} ||\beta^{(\tilde{i})}||_1-\min_{i\in[n]} ||\beta^{(i)}||_1\Big).$$
\item For every $j\in [m]$, $v^{(j)}\geq \mathbf{0}_{p}$ is a $\iota$-sparse vector with $\iota \sim \text{Uni}(\{\lceil \frac{p}{10}\rceil,\dots,p\})$. Positions at which $v^{(j)}_l >0$ were sampled uniformly at random on $[p]$ and were assigned, uniformly at random, a value from $[10]$. Introducing  parameter $\zeta>0$,\\ we decided to set \begin{equation}
W^{(j)} = \zeta\cdot ||v^{(j)}||_1/p>0.
\label{eq:difficulty_param}
\end{equation}

\item According to the classical investment framework, the search space $\mathcal{U}$ takes the form of the $p$-dimensional simplex of radius $R > 0$, $$\mathcal{U} := \{u \in \mathbb{R}^p\,|\,\langle \mathbf{1}_p, u \rangle = R, \,u\geq \mathbf{0}_p\}.$$
\end{itemize}
Finally, we note that problem \eqref{eq:exp_pl_opt} fulfills Assumption \ref{A1} since $\mathcal{U}$ is bounded. \\

\noindent Our \eqref{eq:exp_pl_opt} problems can be described using seven parameters $(n,m,I,\zeta,p,R,\omega)$.
We have investigated three different settings, namely \emph{nominal} (\emph{hard/easy}) and \emph{big-}$n$, for which we detail the \emph{nominal} value of the six first parameters in Table \ref{tab:num_exp_parameters}.
\\

\begin{table}[ht] 
\centering
\renewcommand*{\arraystretch}{1.5}
\begin{tabular}{c||c|c|c}
\hline
\textbf{dimensions / parameters} & nominal (hard) & nominal (easy) & big-$n$\\
    \hline
$(n,m)$ & \multicolumn{2}{c}{$(5 \cdot 10^2,10^4)$}  & $(10^3,10^3)$\\
\hline
$I$ & $1$ & $10$ & $1$ \\ 
\hline
$\zeta$ & \multicolumn{3}{c}{$\{10^{-1},1,5\}$} \\
\hline
$(p,R)$ &$(3\cdot10^2,10^2)$ &\multicolumn{2}{c}{$(10^2,10)$}\\
    \hline
\end{tabular}
\vspace{5pt}
\caption{Nominal values of the \eqref{eq:exp_pl_opt} parameters.}
\label{tab:num_exp_parameters}
\end{table}
\newpage 
\paragraph*{Benchmark procedure} Throughout the sequel, as we benchmark the three studied algorithms (\texttt{ULO}, \texttt{RAM}, \texttt{ES}) on a new instance problem, we run times methods \texttt{ULO} and \texttt{RAM} $N_{\text{rep}} = 15$ with common $N_{\text{rep}}$ initial pieces $(\hat{i}_1,\dots,\hat{i}_{N_{\text{rep}}}) \in [n]^{N_{\text{rep}}}$. 
We have fixed a time limit of 4 minutes for every repetition. We have implemented a log routine within the algorithms to record tuples $(t,\hat{F},\check{F})$
anytime an upper/lower bound candidate is produced at elapsed time $t>0$. We let \texttt{ES} finish until it solves \eqref{eq:exp_pl_opt} by computing the values of the $n$ convex subproblems $\{\nu^{(i)}\}_{i \in [n]}$, as in \eqref{eq:enumeration_scheme_ref}. For any $i\in [n]$, let $t^{(i)}$ depict the time taken to get $\nu^{(i)}$. Then, it is easy to simulate runs of \texttt{ES} without any further numerical experiment by simply considering any bijection of $\varrho : [n]\to [n]$. Each of these bijections induces a permutation of $[n]$ for which we can collect tuples 
\begin{equation}
(t,\hat{F}) = \Bigg(\sum_{k'=1}^{k}\,t^{(\varrho(k'))}, \min_{k'=1,\dots,k} \, \nu^{(\varrho(k'))} \Bigg) \label{eq:evolution}
\end{equation}
at each iteration $k \in [n]$. Therefore, in what concerns algorithm \texttt{ES}, we could afford to sample $10^3$ permutations acting like as many repetitions. 

\paragraph*{Aggregation} In order to compare the algorithms, we defined $N_{\text{ts}} \in \mathbb{N}$ timestamps $0<\tau^{(1)} <\dots< \tau^{(N_{\text{ts}})}$ corresponding to elapsed wall-clock times. Instead of the aforementioned $(t,\hat{F},\check{F})$ involving various unpredictable times $t$, we search for new values $(\tau,\hat{F}_{\text{new}},\check{F}_{\text{new}})$ that will ressemble the most the actual recorded tuples, so that they share a common temporal basis. Thereby, for every setting tested and for every recorded repetition of a given algorithm, we collected the values $(\hat{F},\check{F})$ obtained as late as possible before each timestamp. That is, we picked values of $\hat{F}$ and $\check{F}$ obtained during that repetition corresponding to the last true time $t \leq \tau$. Accordingly, we set $\hat{F}_{\text{new}}\gets \hat{F}$ (respectively $\check{F}_{\text{new}}\gets \check{F}$). We recall that, for algorithms \texttt{ES} and \texttt{RAM}, $\check{F}=-\infty$ unless every piece has been visited, i.e. $k = n$ in \eqref{eq:evolution}. Finally, we take the mean of these new value traces $(\tau,\hat{F}_{\text{new}},\check{F}_{\text{new}})$ to represent the expected behaviour of \texttt{ULO}, \texttt{RAM} or \texttt{ES}.
\remark We point out one minor drawback of the above procedure. For early timestamps, it could happen that for some repetitions, no tuples were recorded yet. Hence, an averaged guarantee (e.g. a primal gap) plotted for a given algorithm would be based on less than $N_{\text{rep}}$ in such case and, eventually, we could observe a deterioration of that guarantee at the next timestamp if ever the new repetitions coming into play present worse $(\hat{F},\check{F})$ values than the average of the previous timestamp.

\newpage 
\vspace{-50pt}
\begin{figure}[H]%
    \hspace{-50pt}
    \vspace{-20pt}
    \subfloat{{\includegraphics[scale=0.4]{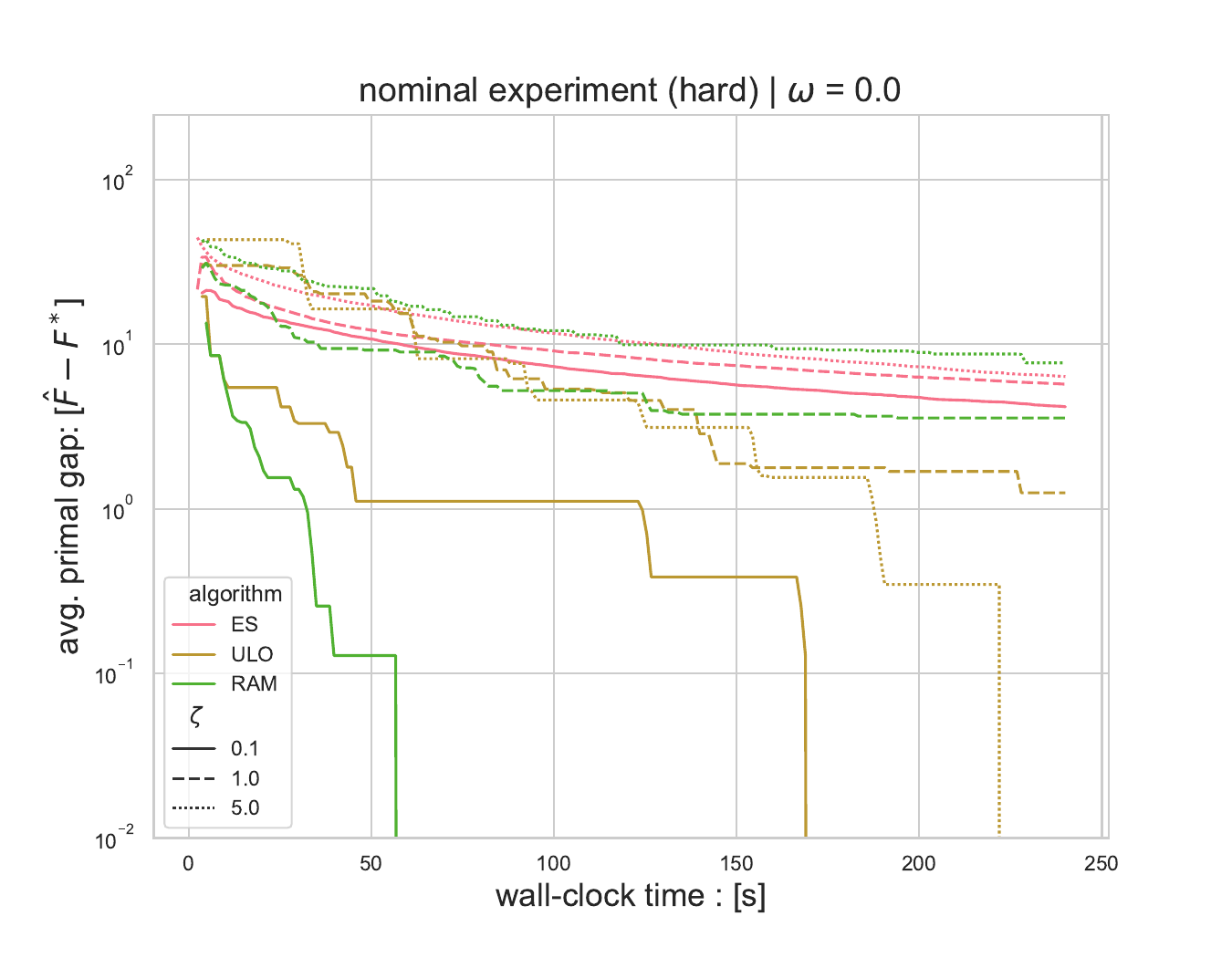} }}%
    \subfloat{{\includegraphics[scale=0.4]{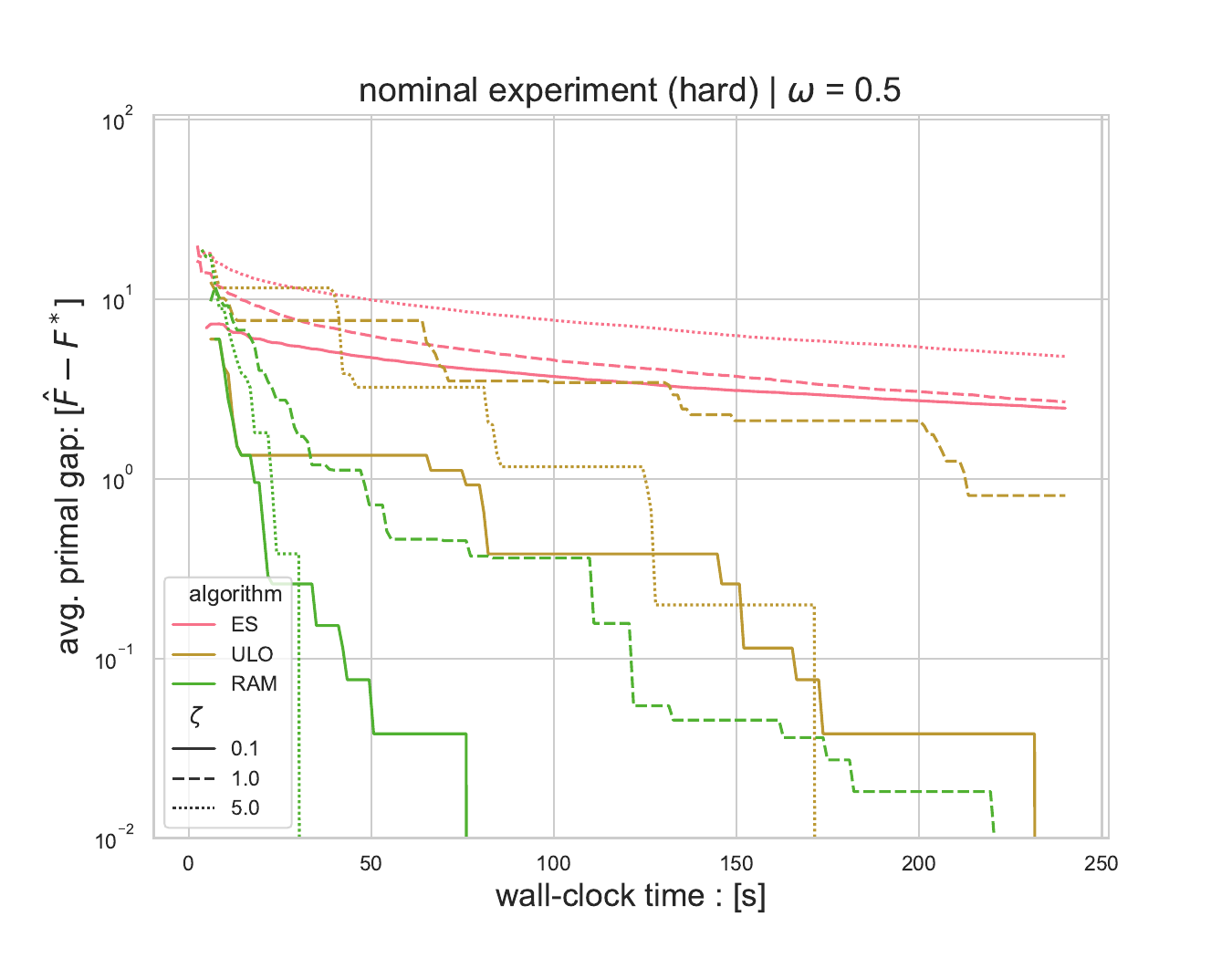} }}% 
\label{fig:nominal_experiment}%
\vspace{-20pt}
\end{figure}
\begin{figure}[H]%
    \hspace{-50pt}
    \vspace{-40pt}
    \subfloat{{\includegraphics[scale=0.4]{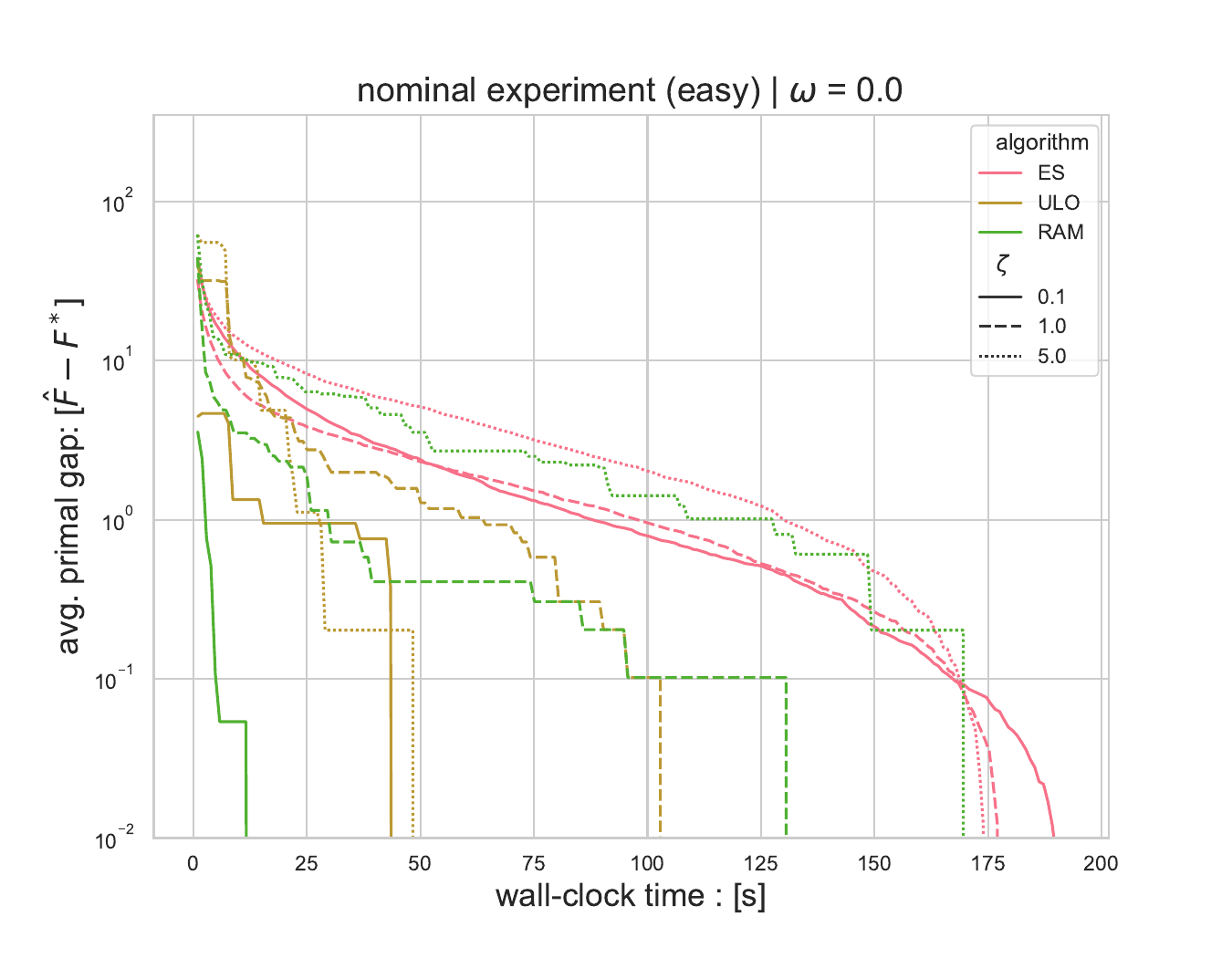} }}%
    \subfloat{{\includegraphics[scale=0.4]{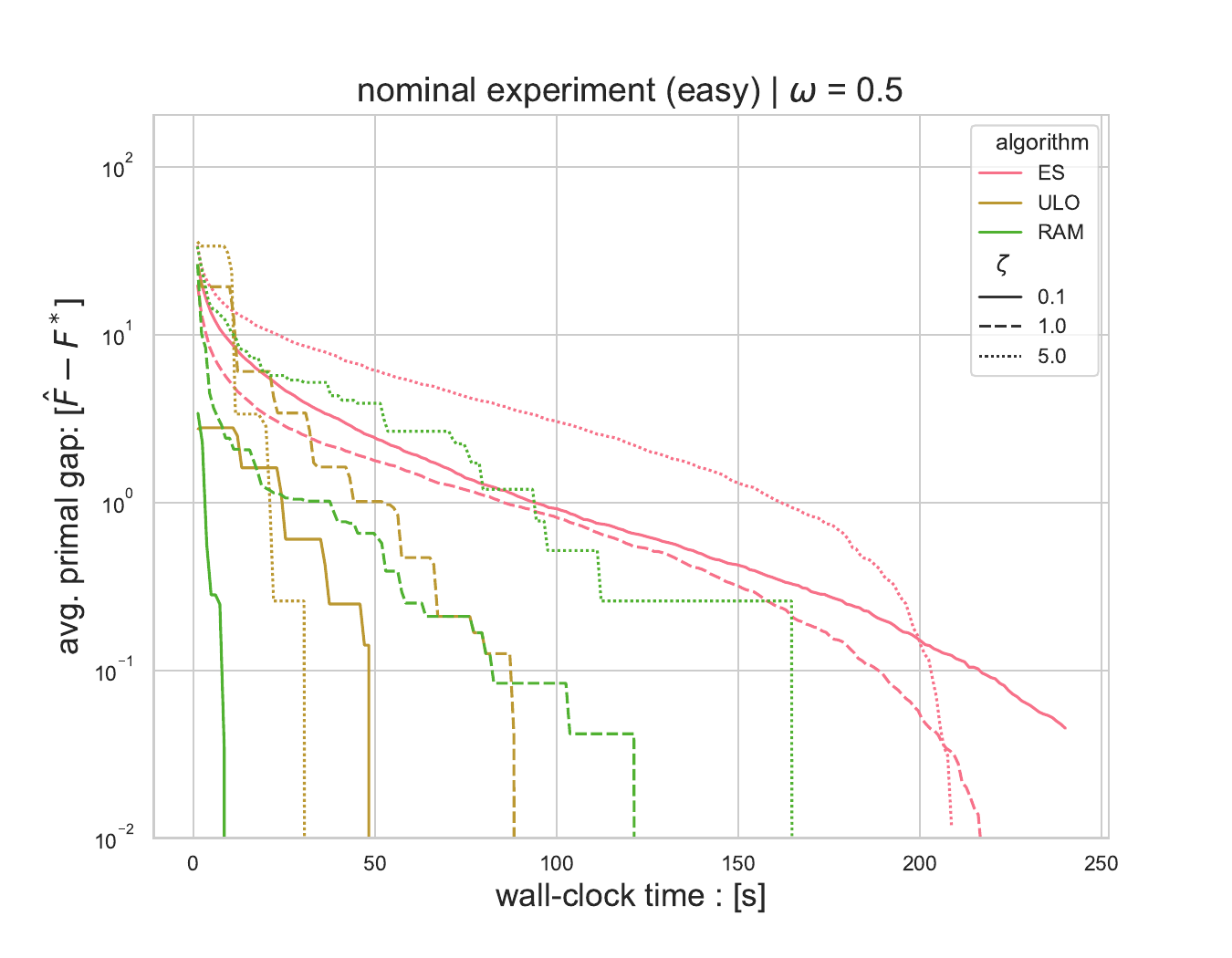} }}% 
\label{fig:nominal_experiment2}%
\end{figure}
\begin{figure}[H]%
    \hspace{-50pt}
    \vspace{-30pt}
    \subfloat{{\includegraphics[scale=0.4]{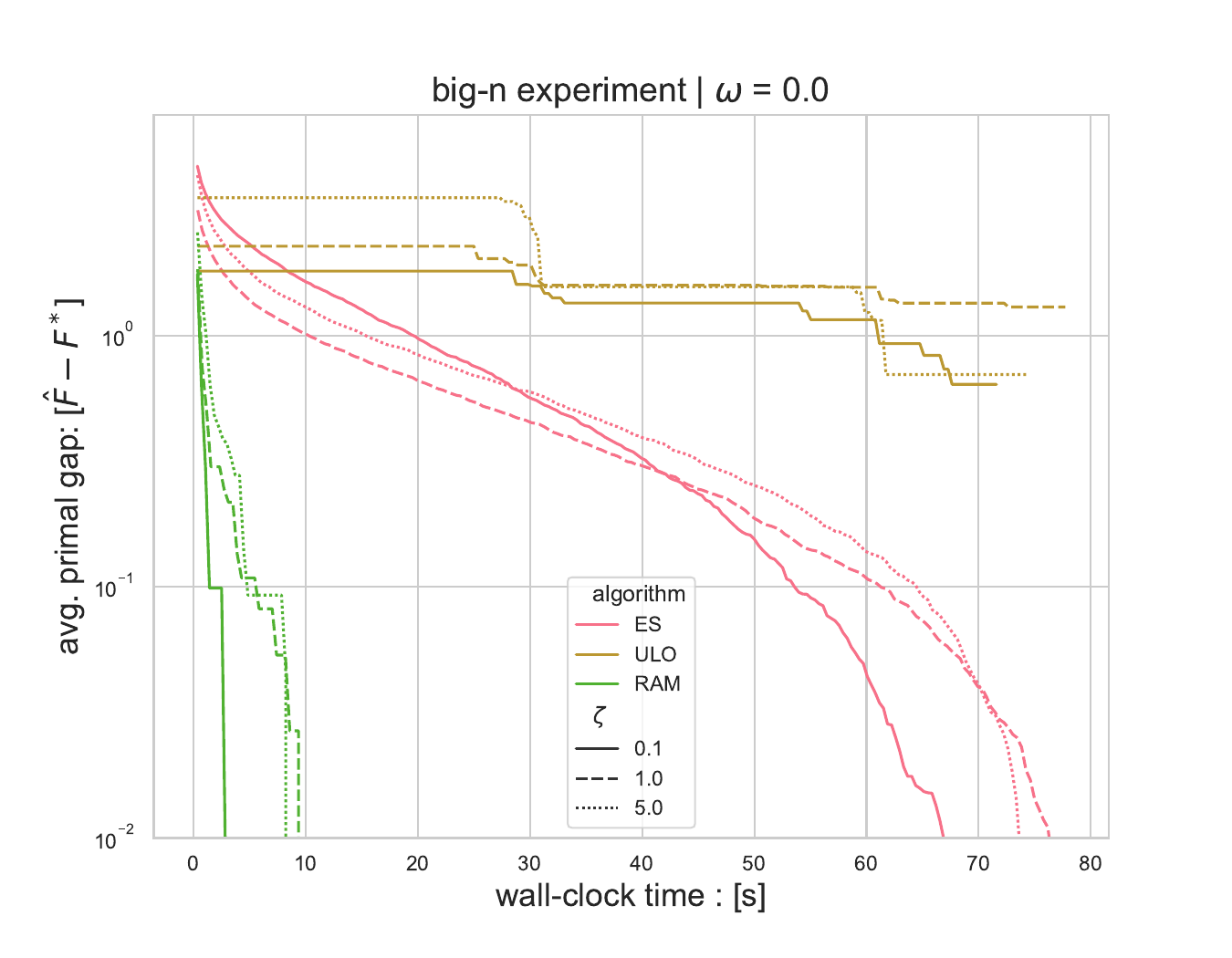} }}%
    \subfloat{{\includegraphics[scale=0.4]{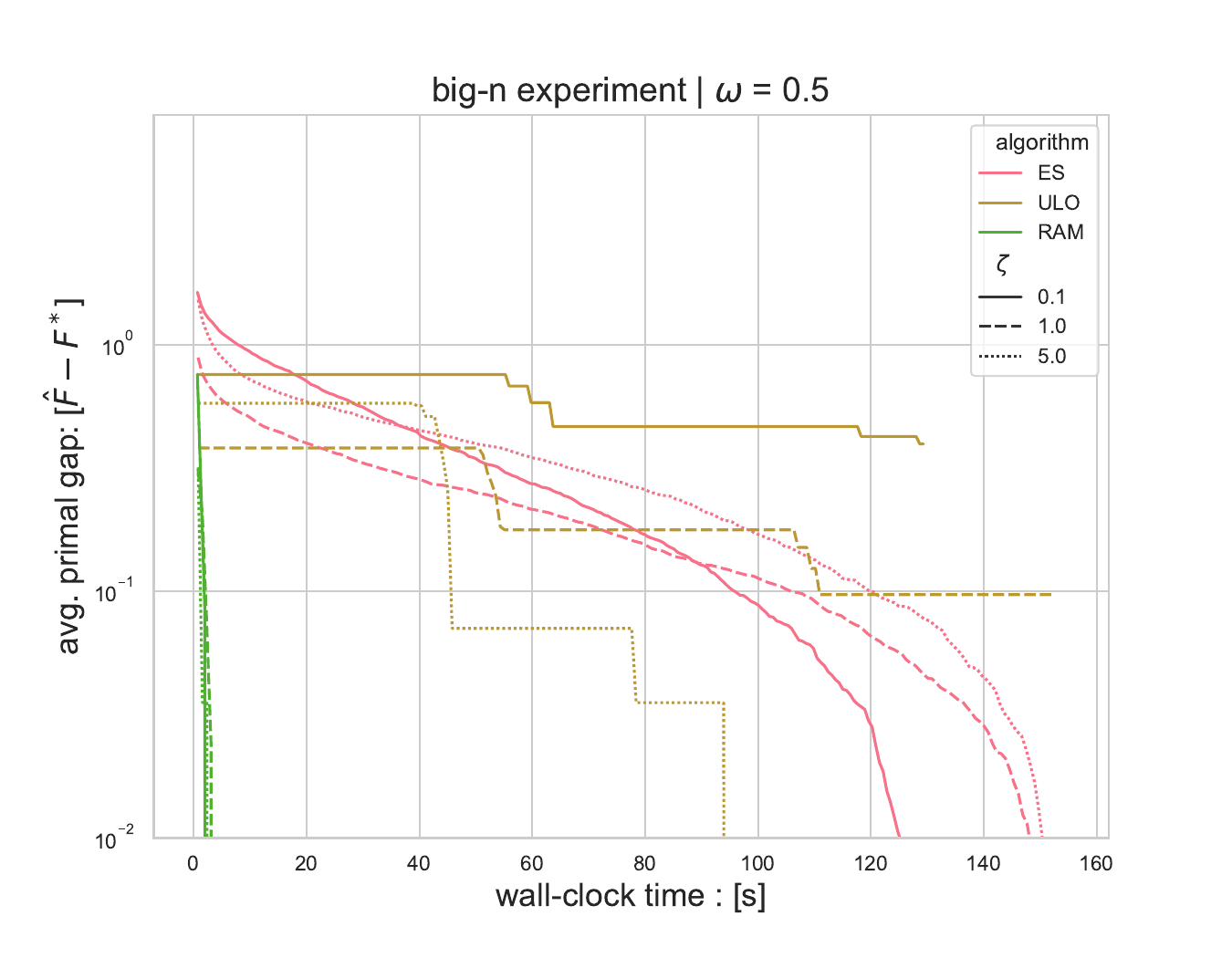} }}% 
\vspace{15pt}
\caption{\eqref{eq:exp_pl_opt} experiments: benchmarks of \texttt{ULO}, \texttt{RAM} \& \texttt{ES} strategies.}
\label{fig:bign_experiment}%
\end{figure}

\newpage
\paragraph*{Results} We subdivide the analysis of results in two parts. In the first part, we make generic comments based on raw pictures. In the second, we document \texttt{ULO}'s behaviour. %We already point out that at this stage that \texttt{ULO} is good at terminating soon when some tolerance $\tilde{\epsilon} > 0$ (see Algorithm \ref{alg:ulo}) is set \textit{a priori}. 

\subparagraph*{\underline{Global analysis.}} Graphs in Figure \ref{fig:bign_experiment} clearly showcase the advantages of phase (a) within Algorithm \ref{alg:ulo} (\texttt{ULO}), especially in the most constrained setting, i.e.\@ $\zeta = 10^{-1}$. Indeed, for every investment condition $(v^{(j)},W^{(j)})$ with $j \in [m]$, inequality $\langle v^{(j)},u\rangle \geq 0$ holds for any $u \in \mathcal{U}$. Therefore, with small values of $\zeta$ in \eqref{eq:difficulty_param}, we bound the quantities $\langle v^{(j)},u\rangle$ from above more strictly, reducing the effective feasible domain for variable $u$ and, as a byproduct, the minimizers $x^*(i)$ of several \emph{pieces} might become identical. Yet, their values $\nu^{(i)}$ presumably differ. As a consequence, \emph{pieces} tend to be more connected in the sense of \eqref{eq:graph_E}. Algorithm \texttt{RAM} can fully exploit that structure since it skips phase (b) from \texttt{ULO} to immediately relaunch phase (a) at a new \emph{piece} $\hat{i} \in [n]$ as soon as it gets stuck (possibly at a \emph{local optimum}). In these favorable situations, many starting \emph{pieces} $\hat{i}$ would be on a connected path, according to the underlying DAG $G([n],E)$ representation of problem \eqref{eq:example_pl_opt}, towards a globally optimizing \emph{piece} $i^* \in [n]$, i.e. $F^* = \nu^{(i^*)}$. \texttt{ULO} also profits from this phenomenon but spends a substantial time in phase (b) to obtain lower bounds certificates. Yet, in the \textit{nominal (hard/easy)} experiments, despite its global optimization focus (unlike \texttt{RAM}), \texttt{ULO} sometimes beats \texttt{RAM} and largely outperforms \texttt{ES}. More precisely, \texttt{RAM} is the fastest to reach any level $0<\epsilon \leq 5\cdot 10^{-2}$ of primal gap $\hat{F}-F^*\leq \epsilon$ for the \emph{nominal (hard)} experiment but \texttt{ULO} takes the win in the \emph{nominal (easy)} experiment when $\zeta=1$ and $\zeta=5$.
Phase (b) sometimes turns out to be a burden. Considering settings like the \textit{big-n} experiment wherein $n=m=10^3$, phase (b)'s relative cost becomes prohibitive if several iterations of \texttt{ULO} ($k = 1,\dots,K$) are run with substantial sizes for $S_k$ in Algorithm \ref{alg:ulo}. This is unfortunately the case on the two last graphs of Figure \ref{fig:bign_experiment} but as we explain hereafter, \texttt{ULO} still performs relatively well.

\subparagraph*{\underline{\texttt{ULO} detailed analysis.}}  On Figure \ref{fig:bign_experiment2_details}, one can observe the evolution of lower bounds $\check{F}$ produced by \texttt{ULO} on the go. The two last graphs allow us to nuance our last statements about \texttt{ULO}'s relative failure on experiment \textit{big-n}. Indeed, we can see that if $\tilde{\epsilon} = 5\cdot 10^{-2}$, i.e. $5\%$ relative accuracy is tolerated in Algorithm \ref{alg:ulo}, then \texttt{ULO} exits with optimality guarantees after $63[s]$ on average when $(\omega,\zeta) = (0,1)$ and exits with even better $1\%$ relative accuracy after roughly $90[s]$ when $(\omega,\zeta) = (1/2,5)$. In comparison, \texttt{ES} terminates with its (exact) optimality guarantees after $76[s]$ and $150[s]$ respectively for equivalent settings. In what concerns the \textit{nominal} experiments, upper and lower bounds both progress relatively smoothly with time making it quite predictable to know when \texttt{ULO} will reach a given level of $\tilde{\epsilon}$ relative accuracy.\\

\noindent We end this discussion by commenting on the value of \eqref{eq:exp_pl_opt} with hyper-parameter $\omega$. Although not displayed on Figure \ref{fig:bign_experiment}, one can notice on the first two graphs of Figure \ref{fig:bign_experiment2_details} that the value of $F^*$ on the right-hand side ($\omega=1/2$) is larger than left-hand side's ($\omega=0$). This was of course expected since problem's data was common to both situations and the parameters $(n,m,I,\zeta,p,R)$ were fixed. More interestingly, as $\omega$ grows up to $1$, the weight put on the common max-affine (convex) part of \emph{pieces} (see description of $f^{(i)}$ in Table \ref{tab:polp_prog}) becomes more significant. Eventually, \eqref{eq:exp_pl_opt} becomes a simple linear program and its intrinsic enumerative complexity with $n$ vanishes. We finally underline that \texttt{ULO} did, on average, take less time on the more convex case ($\omega = 1/2$) to make its upper bound $\hat{F}$ meet $F^*$. It took around $220[s]$ for the case $\omega=0$ against $175[s]$ for $\omega=1/2$.
 
\newpage 
\vspace{-50pt}
\begin{figure}[H]%
    \hspace{-50pt}
    \vspace{-20pt}
    \subfloat{{\includegraphics[scale=0.4]{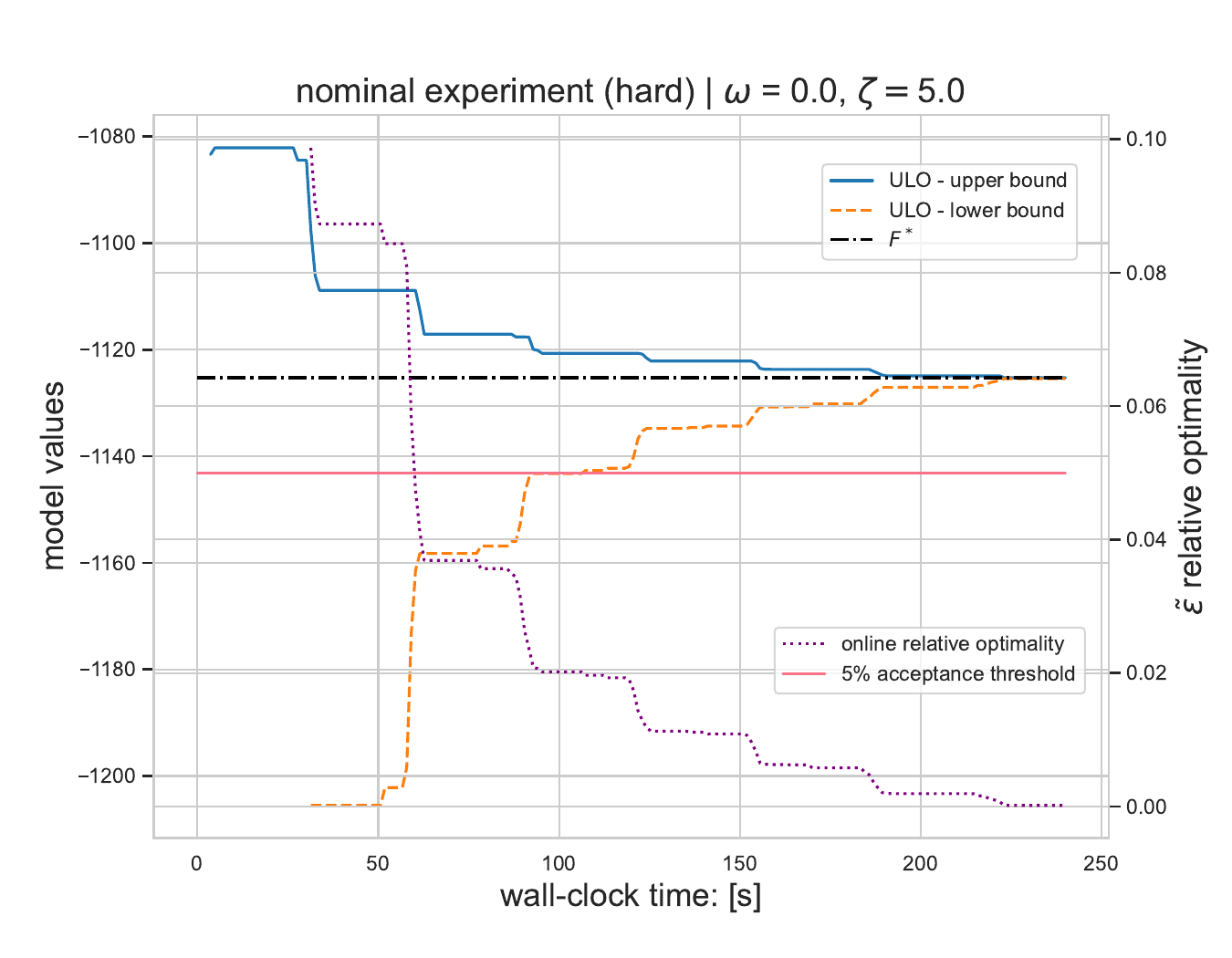} }}%
    \subfloat{{\includegraphics[scale=0.4]{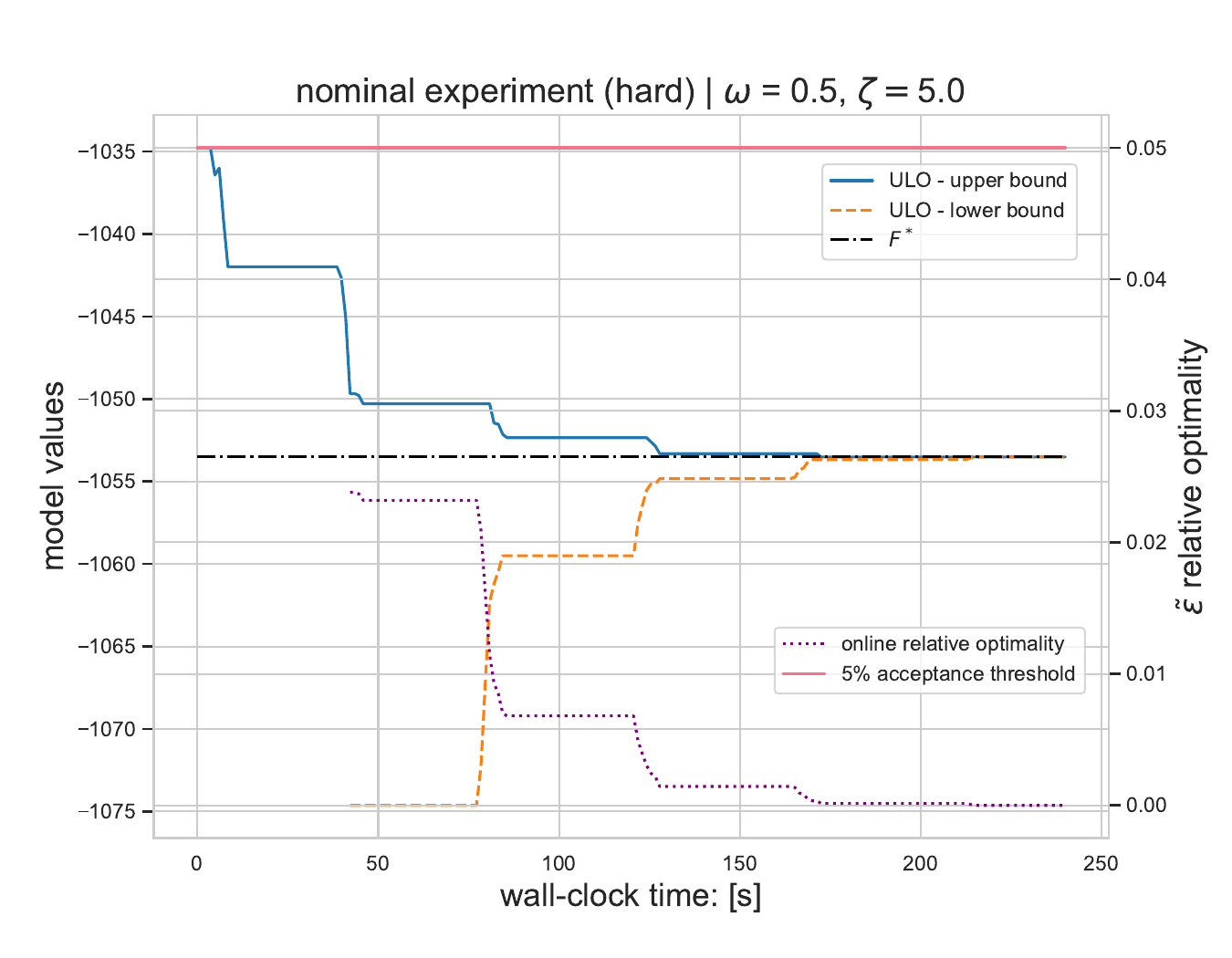} }}% 
\label{fig:bign_experiment_details}%
\vspace{-20pt}
\end{figure}
\begin{figure}[H]%
    \hspace{-50pt}
    \vspace{-40pt}
    \subfloat{{\includegraphics[scale=0.4]{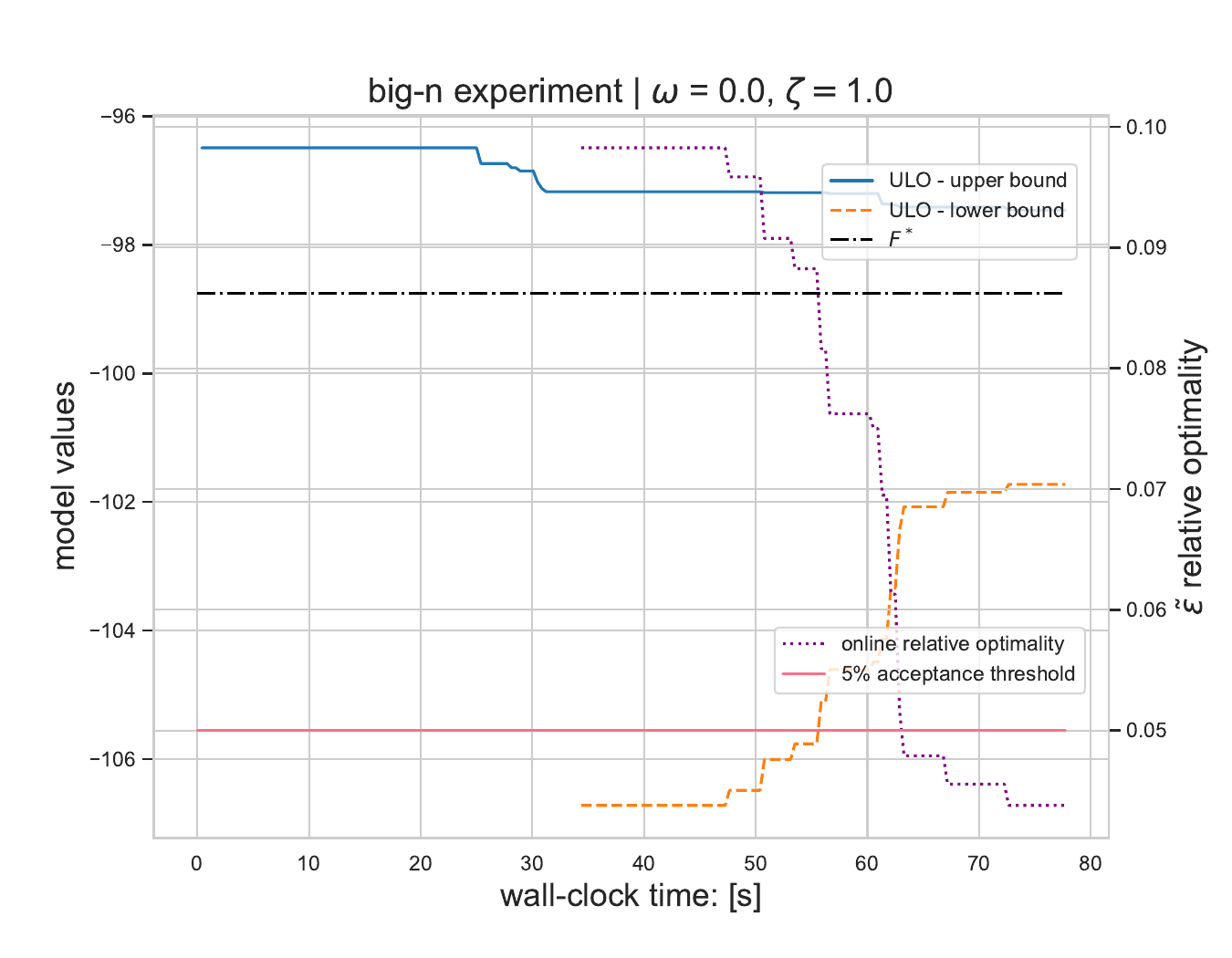} }}%
    \subfloat{{\includegraphics[scale=0.4]{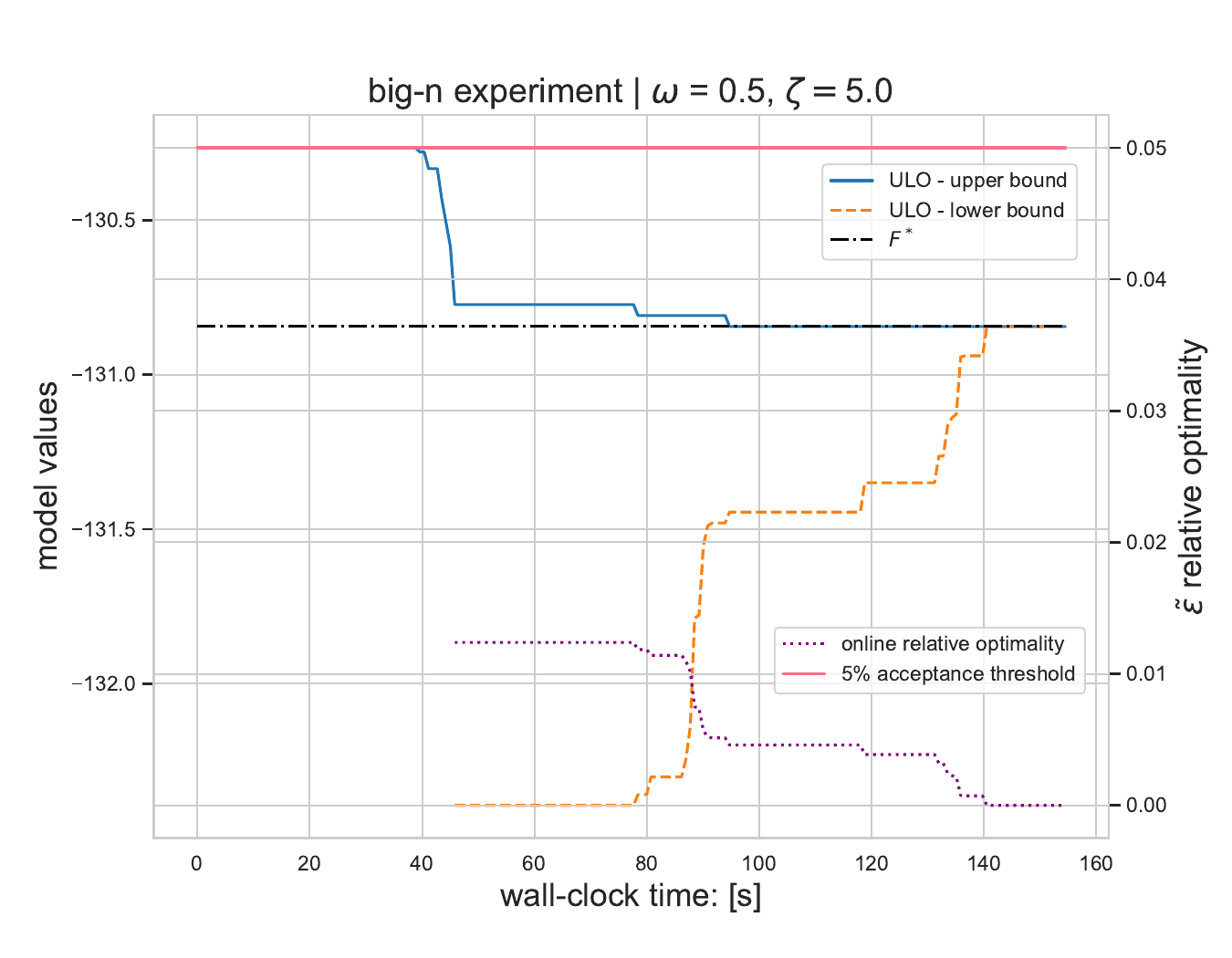} }}% 
\vspace{15pt}
\caption{\eqref{eq:exp_pl_opt} experiments: details of \texttt{ULO} optimality certificates.}
\label{fig:bign_experiment2_details}%
\end{figure}

\section{Conclusion}

 In this paper, we have introduced and motivated the framework of \emph{minimum structured} optimization problems, dubbed as \eqref{eq:min_problem}. Although solvable in a straightforward fashion by a simple \texttt{enumeration scheme}, we have shown how global upper/lower models could be of help to significantly reduce its computational complexity. Based on these models, we have devised and analyzed a new global optimization algorithm, \texttt{ULO}, for which exactness guarantees are provided. We have first conducted theoretical experiments to assess when \texttt{ULO} should exhibit a better performances than the \texttt{enumeration scheme}. Then we have also tested \texttt{ULO} and another local-search heuristic, on a practical optimization task. The results are very convincing, \texttt{ULO} is definitely the algorithm of choice for instances of \eqref{eq:min_problem} linked to bigger ratios $m/n$, as is typically the case for \emph{optimistic robust} linear programs (Example \ref{example:orlp}). We leave the possible parallelization of phase (a) and (b) in \texttt{ULO} with asynchronous communication steps for further research, which could enhance even more its practical performance.
 % further research directions include

\newpage 

\vspace{-5pt}
\section*{Disclosure statement}
\vspace{-10pt}
No potential conflict of interest was reported by the author(s).\vspace{-10pt}
\section*{Funding}
\vspace{-10pt}
\noindent 
Guillaume Van Dessel is funded by the UCLouvain university as a teaching assistant. 
\vspace{-17pt}
\bibliographystyle{abbrv}
\bibliography{references.bib}

\end{document}